\documentclass[10pt]{article}

\usepackage{url}
\usepackage{mathtools}
\usepackage{amssymb}
\usepackage{amsthm}
\usepackage{mathrsfs}
\usepackage{empheq}
\usepackage{latexsym}
\usepackage{enumitem}
\usepackage{eurosym}
\usepackage{dsfont}
\usepackage{appendix}
\usepackage{color} 
\usepackage[unicode]{hyperref}
\usepackage{frcursive}
\usepackage[utf8]{inputenc}
\usepackage[T1]{fontenc}
\usepackage{geometry}
\usepackage{multirow}
\usepackage{lmodern}
\usepackage{anyfontsize}
\usepackage{pgfplots}
\usepackage{upgreek}
\usepackage{comment}
\usepackage{cleveref}
\usepackage{array}
\usepackage{scalerel}

\usepackage[textwidth=0.75in]{todonotes}

\usepackage{natbib}
\setcitestyle{numbers,open={[},close={]},longnamesfirst}

\definecolor{red}{rgb}{0.7,0.15,0.15}
\definecolor{green}{rgb}{0,0.5,0}
\definecolor{blue}{rgb}{0,0,0.7}
\hypersetup{colorlinks, linkcolor={red},citecolor={green}, urlcolor={blue}}
			
\makeatletter \@addtoreset{equation}{section}

\newtheorem{theorem}{Theorem}[section]
\newtheorem{assumption}{Assumption}

\newtheorem*{assumption*}{Assumption}

\newtheorem*{corollary*}{Corollary}

\newtheorem{lemma}[theorem]{Lemma}
\newtheorem*{lemma*}{Lemma}
\newtheorem{proposition}[theorem]{Proposition}

\newtheorem{definition}[theorem]{Definition}
\newtheorem{remark}[theorem]{Remark}
\newtheorem*{theorem*}{Theorem}

\DeclareUnicodeCharacter{014D}{\=o}
\def \E{\mathbb{E}}
\def \F{\mathbb{F}}
\def \G{\mathbb{G}}
\def \H{\mathbb{H}}
\def \I{\mathbb{I}}

\def \L{\mathbb{L}}
\def \M{\mathbb{M}}

\def \P{\mathbb{P}}

\def \R{\mathbb{R}}
\def \S{\mathbb{S}}

\def\Ac{{\cal A}}
\def\Bc{{\cal B}}
\def\Cc{{\cal C}}
\def\Dc{{\cal D}}

\def\Fc{{\cal F}}
\def\Gc{{\cal G}}
\def\Hc{{\cal H}}

\def\Jc{{\cal J}}

\def\Lc{{\cal L}}
\def\Mc{{\cal M}}
\def\Nc{{\cal N}}
\def\Oc{{\cal O}}
\def\Pc{{\cal P}}

\def\Sc{{\cal S}}
\def\Tc{{\cal T}}
\def\Uc{{\cal U}}
\def\Vc{{\cal V}}

\def\Xc{{\cal X}}
\def\Yc{{\cal Y}}
\def\Zc{{\cal Z}}

\newcommand\vc{
  \mathchoice
    {{\scriptstyle\mathcal{V}}}
    {{\scriptstyle\mathcal{V}}}
    {{\scriptscriptstyle\mathcal{V}}}
    {\scalebox{.7}{$\scriptscriptstyle\mathcal{V}$}}
  }

\DeclareMathAlphabet{\pazocal}{OMS}{zplm}{m}{n}

\def\Hf{{\mathfrak H}}

\def\Mf{{\mathfrak M}}

\def\Tf{{\mathfrak T}}
\def\Yf{{\mathfrak Y}}

\def\eps{\varepsilon}
\def\d{{\mathrm{d}}}
\def\1{\mathbf{1}}
\def\e{{\mathrm{e}}}

\def \timesr{{\rm \times}}

\DeclareMathOperator*{\sgn}{sgn}

\DeclareMathOperator*{\Prob}{Prob} 
\let\ae\relax

\def\as{{\rm -a.s.}}
\def\ae{{\rm -a.e.}}

\def\Tr{{\rm Tr }}

\renewcommand{\|}{\Vert}
\renewcommand{\t}{\top}

\def\Ho{\overline{\H}^{_{\raisebox{-1pt}{$ \scriptstyle 2,2$}}} }

\makeatletter
\def\moverlay{\mathpalette\mov@rlay}
\def\mov@rlay#1#2{\leavevmode\vtop{%
   \baselineskip\z@skip \lineskiplimit-\maxdimen
   \ialign{\hfil$\m@th#1##$\hfil\cr#2\crcr}}}
\newcommand{\charfusion}[3][\mathord]{
    #1{\ifx#1\mathop\vphantom{#2}\fi
        \mathpalette\mov@rlay{#2\cr#3}
      }
    \ifx#1\mathop\expandafter\displaylimits\fi}
\makeatother

\allowdisplaybreaks

\title{A unified approach to well--posedness of type--I backward stochastic Volterra integral equations\footnote{The authors would like to thank Tianxiao Wang and Yushi Hamaguchi for useful comments on a earlier version of this document. The authors gratefully acknowledge the support of the ANR project PACMAN ANR-16-CE05-0027, the PGIF project ``Massive entry of renewable energy in Chile: operation, storage and intermittency'' and the project ``Massive entry of renewable energy: operation, storage and intermittency'' funded by the ``Make Our Planet Great Again'' initiative of the Thomas Jefferson Fund. } }

\author{Camilo {\sc Hern\'andez} \footnote{Columbia University, IEOR department, USA, camilo.hernandez@columbia.edu. Author supported by the CKGSB fellowship.} \and Dylan {\sc Possama\"{i}} \footnote{Columbia University, IEOR department, USA, dp2917@columbia.edu}}
 
\date{\today}

\begin{document}
\maketitle
\begin{abstract}
We study a novel general class of multidimensional type--I backward stochastic Volterra integral equations. Toward this goal, we introduce an infinite dimensional system of standard backward SDEs and establish its well--posedness, and we show that it is equivalent to that of a type--I backward stochastic Volterra integral equation. We also establish a representation formula in terms of non--linear semilinear partial differential equation of Hamilton--Jacobi--Bellman type. As an application, we consider the study of time--inconsistent stochastic control from a game--theoretic point of view. We show the equivalence of two current approaches to this problem from both a probabilistic and an analytic point of view.

\end{abstract}

\section{Introduction}\label{Section:introduction}

This paper is concerned with introducing a unified method to address the well--posedness of backward stochastic Volterra integral equations, BSVIEs for short. BSVIEs are regarded as natural extensions of backward stochastic differential equations, BSDEs for short. On a complete filtered probability space $(\Omega, \Gc, \G, \P)$, supporting an $n$--dimensional Brownian motion $B$, and denoting by $\G$ the $\P$--augmented natural filtration generated by $B$, one is given data, that is to say a $\Gc_T$--measurable random variable $\xi$, and a mapping $g$, referred to respectively as the \emph{terminal condition} and the \emph{generator}. A solution to a BSDE is a pair of $\G$--adapted processes $(Y_\cdot,Z_\cdot)$ such that
\begin{align}\label{Eq:BSDE}
Y_t = \xi +\int_t^T g_r(Y_r,Z_r)\d r -\int_t^T Z_r \d B_r,\; t\in [0,T],\; \P\as
\end{align} 
BSDEs of linear type were first introduced by \citet*{bismut1973analyse, bismut1973conjugate} as an adjoint equation in the Pontryagin stochastic maximum principle. Actually, the contemporary work of \citet*{davis1973dynamic}\footnote{Indeed, \cite{davis1973dynamic} was received for publication on October 27, 1971 and it is part of the bibliography of \cite{bismut1973conjugate}.} studied a precursor of a linear BSDE for characterising the value function and the optimal controls of stochastic control problems with drift control only. In the same context of the stochastic maximum principle, BSDEs of linear type are present in \citet*{arkin1979necessary}, \citet*{bensoussan1983maximum} and \citet*{kabanov1978on}. Remarkably, the extension to the non--linear case is due to \citet*{bismut1978controle}, as a type of Riccati equation, as well as \citet*{chitashvili1983martingale}, and \citet*{chitashvili1987optimal, chitashvili1987optimal2}. Later, the seminal work of \citet*{pardoux1990adapted} presented the first systematic treatment of BSDEs in the general nonlinear case, while the celebrated survey paper of \citet*{el1997backward} collected a wide range of properties of BSDEs and their applications to finance. Among such properties we recall the so--called \emph{flow property}, that is to say, for any $0\leq r\leq T$, 
\[ Y_t(T,\xi)=Y_t(r,Y_r(T,\xi)),\; t\in [0,r], \; \P\as, \text{ and } Z_t(T,\xi))=Z_t(r,Y_r(T,\xi)),\; \d t\otimes \d \P\ae\text{ on } [0,r]\times \Omega,\] where $(Y(T,\xi), Z(T,\xi))$ denotes the solution to the BSDE with terminal condition $\xi$ and final time horizon $T$.\medskip

A natural extension of \eqref{Eq:BSDE} arises by considering a collection of $\Gc_T$--measurable random variables $(\xi(t))_{t\in [0,T]}$, referred in the literature of BSVIEs as the \emph{free term}, as well as a generator $g$. In such a setting, a solution to a BSVIE is a pair $(Y_\cdot ,Z_\cdot^\cdot )$ of processes such that
\begin{align}\label{Eq:typeIIBSVIE}
Y_t=\xi(t)+\int_t^T g_r(t,Y_r, Z_r^t,Z_t^r)\d r -\int_t^T Z_r^t \d B_r, \; t\in [0,T],\; \P\as
\end{align}
Of noticeable interest is the case in which the term $Z_t^r$ is absent in the generator, i.e.
\begin{align}\label{Eq:typeIBSVIE}
Y_t=\xi(t)+\int_t^T g_r(t,Y_r,Z_r^t)\d r -\int_t^T Z_r^t \d B_r,\;  t\in [0,T],\; \P\as
\end{align}

Nowadays \eqref{Eq:typeIBSVIE} and \eqref{Eq:typeIIBSVIE} are referred in the literature as type--I and type--II BSVIEs, respectively. The first mention of such equations is, to the best of our knowledge, due to \citet*{hu1991adapted}. Indeed, in the context of well--posedness of BSDEs valued in a Hilbert space, a prototype of type--I BSVIEs \eqref{Eq:typeIBSVIE} is considered, see the comments following \cite[Remark 1.1]{hu1991adapted}. Two decades passed before a direct consideration of BSVIEs of the form given by \eqref{Eq:typeIBSVIE} was made by \citet*{lin2002adapted}, where the author studied the case $\xi(t)=\xi,$ $t\in [0,T]$, for a $\Gc_T$--measurable $\xi$. The general form of \eqref{Eq:typeIIBSVIE} was first addressed in \citet*{yong2006backward,yong2008well} in the context of optimal control of (forward) stochastic Volterra integral equations (FSVIEs, for short).\medskip

There are significant distinctions between BSDEs and BSVIEs. Nevertheless, a satisfactory concept of solution for such equations can be defined by extrapolating from the theory of BSDEs. In broad terms, a pair $(Y_\cdot , Z_\cdot^\cdot )$ is said to be a solution to a BSVIE, see \cite{yong2008well}, if for each $s\in [0,T)$, the mapping $t\longmapsto (Y_t, Z_t^s)$ is $\G$--adapted on $[s, T]$, $(Y,Z)$ is appropriately integrable and satisfies \eqref{Eq:typeIIBSVIE}. It is also worth pointing out the distinctions between type--I and type--II BSVIEs. As a consequence of the presence of $Z^s_t$ in the generator, to obtain a solution to a type--II BSVIE one has to determine $Z^s_t$ for $(t, s) \in [0, T]^2$, and \eqref{Eq:typeIIBSVIE} alone does not give enough restrictions. Indeed, \cite{yong2008well} showed that an adapted solution to the type--II BSVIE \eqref{Eq:typeIIBSVIE} may, in general, not be unique. This is in contrasts with type--I BSVIEs, where it suffices to determine $Z_t^s$ for $(t, s) \in [0, T]^2,$ $0\leq s\leq t\leq T$. Moreover, without additional assumptions, a solution to a general type--II BSVIE does not satisfy the flow property. \medskip

Since their introduction, BSVIEs have been extended to much more general frameworks than the one presented above. Hence, \citet*{wang2019bsvie} studies the case of random Lipschitz data; \citet*{wang2007nonlipschitz} and \citet*{shi2012solvability} deal with general non--Lipschitz data; \citet*{coulibaly2019backward} study time--delayed generators; \emph{mean--field} BSVIEs are considered in \citet*{shi2013meanfield}; \citet*{lu2016backward}, \citet*{hu2019linear} and \citet*{popier2020backward} studied extensions to general filtrations and the case where $B$ is replaced by more general processes; \citet*{djordjevic2015backward, djordjevic2013on} were interested in perturbed BSVIEs, i.e. when the coefficients depend additively on small perturbations; BSVIEs in Hilbert spaces have been investigated in \citet*{anh2006backward}, \citet*{ anh2011regularity}, and 
\citet*{ren2010solutions}; and an analysis of numerical schemes for BSVIEs has been proposed in \citet*{bender2013discretization}. There is also a wide spectrum of applications of BSVIEs. Hence, dynamic risk measures have been considered in \citet*{yong2007continuous}, \citet*{wang2010symmetrical,wang2013class}, \citet*{wang2018recursive} and \citet*{agram2019dynamic}. \citet*{kromer2017differentiability} also studied the question of dynamic capital allocations via BSVIEs. \citet*{wang2010maximum} dealt with a risk minimisation problem by means of the maximum principle for FBSVIEs, while the optimal control of SVIEs and BSVIEs via the maximum principle has been studied in \citet*{chen2007linear}, \citet*{wang2018linear}, \citet*{agram2017optimal,agram2018new}, \citet*{shi2015optimal}, \citet*{shi2019backward}, see also \citet*{wei2013optimal} for the case with state constraints.

\medskip
Since their first appearance, a natural and non--trivial question for BSVIEs has been that of the regularity in time of their solutions. The best known probabilistic results for general type--II BSVIEs guarantee the regularity of the solutions in an $\L^p$ sense only, see \citet*{wang2012Lp} and \citet*{ li2014existence}. Nevertheless, analytic results via a representation formula, guarantee the pathwise regularity of a solution to type--I BSVIEs, see \citet*{wang2019backward} and \citet*{wang2020path} for results regarding the representation of BSVIEs in the Markovian and non--Markovian framework, respectively. We also highlight that \citet*{overbeck2017path} surveyed path--dependent BSVIEs and their path regularity. In general, type--I BSVIEs are known to be much more amenable to the analysis, for example \cite{wang2007nonlipschitz} is able to establish the regularity of type--I BSVIEs by probabilistic methods. \medskip

\medskip

Out of the class of processes described by BSVIEs, a broader family than that of standard type--I BSVIEs \eqref{Eq:typeIBSVIE} is known to arise in the study of time--inconsistent control problems. Recently, \citet*{agram2020reflected} studied reflected backward stochastic Volterra integral equations and their relations to a time--inconsistent optimal stopping problem. Earlier connections were suggested in the concluding remarks of \citet*{wang2019backward}. Indeed, BSVIEs provide a probabilistic representations of the system of partial differential equation (PDE, for short) appearing in the study of time--inconsistent optimal control problems, e.g. see \citet*{yong2012time} and \citet*{wei2017time} for PDEs obtained via Pontryagin's and Bellman's principle, respectively. A natural link was then made rigorous independently by \citet*[Section 5]{wang2019time} and \citet*[Lemma A.2.3]{hernandez2020me}. Although following different approaches, their analyses lead to introduce type--I BSVIEs of the form 
\begin{align}\label{Eq:typeIBSVIEextended0}
Y_t=\xi(t)+\int_t^T g_r(t,Y_r, Z_r^t,Z_r^r)\d r -\int_t^T Z_r^t \d B_r,\; t\in [0,T], \; \P\as
\end{align}

These are {\rm BSVIE}s in which the diagonal of $Z$ appears in the generator. We highlight that, until the present work, the only well--posedness results in the literature for type--I BSVIEs \eqref{Eq:typeIBSVIEextended0} are available in \cite{wang2019time} and \cite{hernandez2020me}. Both results hold for the particular case in which the driver $g$ is linear in $Z_r^t$. Indeed, the argument in \cite{wang2019time} follows as a consequence of the representation formula, i.e. an analytic argument via PDEs, and holds in a Markovian setting. On the other hand, the probabilistic argument in \cite{hernandez2020me} holds in the non--Markovian case.\medskip

Likewise, \citet*{hamaguchi2020extended,hamaguchi2020small} studied a time--inconsistent control problem where the cost functional is defined by the $Y$ component of the solution of a type--I BSVIE \eqref{Eq:typeIBSVIE}, in which $g$ depends on a control. Via Pontryagin's optimal principle, the author noticed that the adjoint equations correspond to an \emph{extended} type--I BSVIE, as first introduced in \citet*{wang2020extended} in the context of generalising the celebrated Feynman--Kac formula. An extended type--I BSVIE consists of a pair $(Y^\cdot_\cdot,Z^\cdot_\cdot)$, with appropriate integrability, such that $s\longmapsto Y^s$ is continuous in an appropriate sense for $s\in [0,T]$, $Y^s_\cdot$ is pathwise continuous, $Z^s_\cdot$ is predictable, and 
\begin{align}\label{Eq:typeIBSVIEextended}
Y^s_t=\xi(t)+\int_t^T g_r(s,Y_r^s, Z_r^s,Y_r^r)\d r -\int_t^T Z_r^s \d B_r,\; (s,t) \in [0,T]^2,\; \P\as
\end{align}

We highlight that the noticeable feature of \eqref{Eq:typeIBSVIEextended0} and \eqref{Eq:typeIBSVIEextended} is the appearance of the `diagonal' processes $(Y_t^t)_{t\in[0,T]}$ and $(Z_t^t)_{t\in[0,T]}$, respectively. A prerequisite for rigorously introducing these processes is some regularity of the solution. Indeed, the regularity of $s\longmapsto (Y^s,Z^s)$ in combination with the pathwise continuity of $Y$ and the introduction of a derivative of $Z^s$ with respect to $s$, as first discussed in \cite{hernandez2020me}, make the analysis possible, see \Cref{Remark:wpsystem} for details. \medskip

Put succinctly, type--I BSVIEs, understood in a broader sense than that of \eqref{Eq:typeIBSVIE}, provide a rich framework to address new classes of problems in mathematical finance and control. In the case of time--inconsistent control problems, \eqref{Eq:typeIBSVIEextended0} and \eqref{Eq:typeIBSVIEextended} appear as a consequence of the study of such problems via Bellman's and Pontryagin's principles, respectively. Consequently, in this paper we want to build upon the strategy devised in \cite{hernandez2020me} and address the well--posedness of a general and novel class of type--I BSVIEs. We let $X$ be the solution to a drift--less stochastic differential equation (SDE, for short) under a probability measure $\P$, and $\F$ be the $\P$--augmentation of the filtration generated by $X$, see \Cref{Section:stochasticbasis} for details, and consider a tuple $(Y_\cdot^\cdot, Z_\cdot^\cdot, N_\cdot^\cdot)$, of appropriately $\F$ adapted processes, satisfying
\begin{align}\label{Eq:typeIBSVIEfe}
Y^s_t=\xi(t)+\int_t^T g_r(s,X,Y_r^s, Z_r^s,Y_r^r,Z_r^r)\d r -\int_t^T Z_r^s \d X_r- \int_t^T \d N_r^s,\; (s,t) \in [0,T]^2,\; \P\as
\end{align}

We remark that the additional process $N$ corresponds to a martingale process which is $\P$--orthogonal to $X$. This is a consequence of the fact that we work with a general filtration $\F$. To the best of our knowledge, a theory for type--I BSVIEs, as general as the ones introduced above, remains absent in the literature. Moreover, such class of type--I BSVIEs has only been mentioned in \cite[Remark 3.8]{hamaguchi2020extended} as an interesting generalisation of \eqref{Eq:typeIBSVIEextended}.\medskip

Our approach is based on the following class of infinite dimensional systems of BSDEs, supposed to hold $\P\as$
\begin{align*} 
\begin{split}
\Yc_t&=\xi(T)+\int_t^T h_r(X,\Yc_r,\Zc_r, Y_r^r,Z_r^r)\d r-\int_t^T  \Zc_r \d X_r- \int_t^T \d \Nc_r^s ,\; t \in [0,T],  \\
Y_t^s&=   \eta (s)+\int_t^T  g_r(s,X,Y_r^s,Z_r^s, \Yc_r,\Zc_r) \d r-\int_t^T Z_r^s  \d X_r - \int_t^T \d N_r^s,\; (s,t)\in[0,T]^2.
\end{split}
\end{align*}
where $(\Yc,\Zc,\Nc,Y,Z,N)$ are unknown, and required to have appropriate integrability, see \Cref{Section:infdimsystem} and \Cref{Eq:systemBSDE}.\medskip

We first establish the well--posedness of \eqref{Eq:systemBSDE}, see \Cref{Thm:wp}. For this it is important to be able to identify the proper spaces to carry out the analyses, see \Cref{Remark:wpsystem}. Moreover, we show that, for an appropriate choice of data for \eqref{Eq:systemBSDE}, its well--posedness is equivalent to that of the type--I BSVIE \eqref{Eq:typeIBSVIEfe}, see \Cref{Thm:wpbsvie}. Noticeably, our approach can naturally be specialised to obtain the well--posedness of \eqref{Eq:typeIBSVIE}, \eqref{Eq:typeIBSVIEextended0} and \eqref{Eq:typeIBSVIEextended} in the classic spaces, see \Cref{Remark:wpBSVIE}. Moreover, as our results provide an alternative approach to BSVIEs, it may allow for the future design of new numerical schemes to solve type--I BSVIEs, which to the best of our knowledge, remain limited to \cite{bender2013discretization}. In addition, we recover classical results for this general class of multidimensional type--I BSVIEs. We provide \emph{a priori} estimates, show the stability of solutions as well as a representation formula in terms of a semilinear PDEs, see \Cref{Thm:rep}. Given our multidimensional setting, we refrained from considering comparison results, see \citet*{wang2015comparison} for the one--dimensional case.\medskip

As an application of our results, we consider the game--theoretic approach to time--inconsistent stochastic control problems. We recall this approach studies the problem faced by the, so--called, \emph{sophisticated agent} who aware of the inconsistency of its preferences seeks for \emph{consistent plans}, i.e. equilibria. We show that as a consequence of \Cref{Thm:wpbsvie}, one can reconcile two recent probabilistic approaches to this problem. Moreover, we provide, see \Cref{Thm:eqvpds}, an equivalent result for two earlier analytic approaches, based on semi--linear PDEs. We believe this helps to elucidate connections between the different takes on the problem available in the literature.
\medskip

The rest of the paper is structured as follows. \Cref{Sec:Pstatement} introduces the stochastic basis on a canonical space as well as the integrability spaces necessary to our analysis. \Cref{Section:infdimsystem} precisely formulates the class of infinite dimensional systems of BSDEs \eqref{Eq:systemBSDE}, which is the crux of our approach, and provides the statement of its well--posedness, while the proof is deferred to \Cref{Section:Analysis}. \Cref{Section:BSVIE} introduces the class of type--I BSVIEs which are the main object of this paper, and establishes the equivalence of its well--posedness with that of \eqref{Eq:systemBSDE} for a particular choice of data. \Cref{Section:PDEsandtimeinconsistency} deals with the representation formula for the class of type--I BSVIEs considered, and presents the application of our results in the context of time--inconsistent stochastic control. Finally, \Cref{Section:Analysis} includes the analysis of \eqref{Eq:systemBSDE}.

\section{Preliminaries}\label{Sec:Pstatement}

{\bf Notations:}
we fix a time horizon $T>0$. Given $(E,\|\cdot \|)$ a Banach space, a positive integer $d$, and a non--negative integer $q$, $\Cc^{d}_{q}(E)$ (resp. $\Cc_{q,b}^d(E)$) will denote the space of functions from $E$ to $\R^{d}$ which are $q$ times continuously differentiable (resp. and bounded with bounded derivatives). When $d=1$ we write $\Cc_{q}(E)$ and $\Cc_{q,b}(E)$. For $\phi \in \Cc_{0,q}([0,T]\times E)$ with $q\geq 2$,
if $s \longmapsto \phi(s,\alpha)$ is uniformly continuous uniformly in $\alpha$, we denote by $\rho_{\phi }:[0,T]\longrightarrow \R$ its modulus of continuity. $\partial_\alpha \phi$ and $\partial_{\alpha \alpha}^2 \phi$ denote the gradient and Hessian with respect to $\alpha$, respectively. For $(u,v) \in (\R^p)^2$, $u\cdot v$ will denote their usual inner product, and $|u|$ the corresponding norm. For positive integers $m$ and $n$, we denote by $M_{m,n}(\R)$ the space of $m\times n$ matrices with real entries, and set $M_n(\R):=M_{n,n}(\R)$. For $M\in M_{m,n}(\R)$, $M_{:i}$ and $M_{i:}$ denote the i--th column and row. $\S_n^+(\R)$ denotes the set of $n\times n$ symmetric positive semi--definite matrices, while $\Tr [M]$ denotes the trace of $M\in  M_{m}(\R)$, and $| M|:=\sqrt{\Tr[M^\t M]}$ for $M\in M_{m,n}(\R)$.

\medskip
For $(\Omega, \Fc)$ a measurable space, $\Prob(\Omega)$ denotes the collection of probability measures on $(\Omega, \Fc)$. For a filtration $\F:=(\Fc_t)_{t\in [0,T]}$ on $(\Omega, \Fc)$, $\Pc_{\rm pred}(E,\F)$ (resp. $\Pc_{\rm prog}(E,\F)$, $\Pc_{\rm opt}(E,\F)$, $\Pc_{\rm meas}(E,\F)$) denotes the set of $E$--valued, $\F$--predictable processes (resp. $\F$--progressively measurable processes, $\F$--optional processes, $\F$--adapted and measurable). For $\P\in \Prob(\Omega)$, $\F^\P:=(\Fc_t^\P)_{t\in[0,T]},$ denotes the $\P$--augmentation of $\F$, where for $t\in [0,T]$, $\Fc^\P_t:=\Fc_t\vee \sigma(N^\P)$, where $N^\P:=\{N\subseteq \Omega: \exists B \in \Fc, N \subseteq B \text{ and } \P[B]=0\}$. With this, $\P\in \Prob(\Omega)$ can be extended so that $(\Omega,\Fc, \F^\P,\P)$ becomes a complete probability space, see \citet*[Chapter II.7]{karatzas1991brownian}. $\F^\P_+$ denotes the right limit of $\F^\P$, i.e. $\Fc_{t+}^\P:=\bigcap_{\eps>0} \Fc_{t+\eps}^\P$, $t\in[0,T)$, and $\Fc_{T+}^\P:=\Fc_T^\P$, so that $\F^{\P}_+$ is the minimal filtration that contains $\F$ and satisfies the usual conditions. For $\{s,t\}\subseteq [0,T]$, with $s\leq t$, $\Tc_{s,t}(\F)$ denotes the collection of $[t,T]$--valued $\F$--stopping times.\medskip

For $\P\in\Prob(\Omega)$, $Z\in \Pc_{{\rm pred}}(E,\F)$ and $X$ an $(E,\F,\P)$--semi--martingale, we set $Z\bullet X:=(Z\bullet X_t)_{t\in [0,T]}$, where $Z\bullet X_t:=\int_0^t Z_r   \d X_r$, $t\in[0,T]$. \medskip

\subsection{The stochastic basis on the canonical space}\label{Section:stochasticbasis}

We fix two positive integers $n$ and $m$, which represent respectively the dimension of the martingale which will drive our equations, and the dimension of the Brownian motion appearing in the dynamics of the former. We consider the canonical space $\Xc:=\Cc([0,T],\R^{n})$, with canonical process $X$. We let $\Fc$ be the Borel $\sigma$--algebra on $\Xc$ (for the topology of uniform convergence), and we denote by $\F^o:=(\Fc^o_t)_{t\in[0,T]}$ the natural filtration of $X$. We fix a bounded Borel measurable map $\sigma:[0,T]\times\Xc \longrightarrow \R^{n\times m}$, $\sigma_\cdot(X)\in \Pc_{{\rm meas}}(\R^{n\times m},\F^o)$, and an initial condition $x_0\in\R^n$. We assume there is $\P\in \Prob (\Xc)$ such that $\P[X_0=x_0]=1$ and $X$ is martingale, whose quadratic variation, $\langle X\rangle=(\langle X\rangle_t)_{t\in [0,T]}$, is absolutely continuous with respect to Lebesgue measure, with density given by $\sigma\sigma^\t$. Enlarging the original probability space, see \citet*[Theorem 4.5.2]{stroock1997multidimensional}, there is an $\R^{m}$--valued Brownian motion $B$ with
\[
X_t=x_0+\int_0^t\sigma_r(X_{\cdot \wedge r}) \mathrm{d}B_r,\; t\in[0,T],\; \P\as
\]
We now let $\F:=(\Fc_t)_{t\in[0,T]}$ be the (right--limit) of the $\P$--augmentation of $\F^o$. We stress that we will not assume $\P$ is unique. In particular, the predictable martingale representation property for $(\F,\P)$--martingales in terms of stochastic integrals with respect to $X$ might not hold. 
\begin{remark}
We remark that the previous formulation on the canonical is by no means necessary. Indeed, any probability space supporting a Brownian motion $B$ and a process $X$ satisfying the previous {\rm SDE} will do, and this can be found whenever that equation has a weak solution.
\end{remark}

\subsection{Functional spaces and norms}\label{Section:spacesandnorms}

We now introduce our spaces. In the following, $(\Omega,\Fc, \F,\P)$ denotes the filtered probability space as defined in \Cref{Section:stochasticbasis}. We are given a non--negative number $c$ and $(E,|\cdot |)$ a finite--dimensional Euclidean space, i.e. $E=\R^{k}$ for some non--negative integer $k$ and $|\cdot |$ denotes the euclidean norm. For any $(p,q) \in (1,\infty)^2$, we introduce the spaces

\begin{list}{\labelitemi}{\leftmargin=1em}
\item $\Lc^p(E)$ of $\Fc$--measurable, $E$--valued random variables $\xi$, with $ \| \xi\|_{\Lc^p}^p:= \E^\P[ |\xi|^p]<\infty$;
\item  $\S^p(E)$ of $Y\in \Pc_{\text{opt}}(E,\F)$, with $\P\as$ c\`adl\`ag paths on $[0,T]$, with $ \|Y\|_{\S^p}^p:=\E^\P\Big[ \sup_{t\in[0,T]} |Y_t|^p\Big]<\infty$;

\item  $\L^{q,p}(E)$ of $Y\in \Pc_{\text{opt}}(E,\F)$, with $ \|Y\|_{\L^{q,p}}^p:= \E^\P\bigg[ \bigg( \displaystyle\int_0^T |Y_r|^q\d r \bigg)^{\frac{p}q}\bigg]<\infty$;

\item $\H^{p}(E)$ of $Z \in  \Pc_{\text{\rm pred}}(\E,\F)$, which are defined $\sigma\sigma^\t_t \d t \ae$, with $  \|Z\|_{\H^{p}}^p := \E^\P\Big[ \big( \Tr[ \langle Z \bullet X \rangle_T ] \big)^{\frac{p}{2}}  \Big]  <\infty $;

\item $\M^p(E)$ of martingales $M\in \Pc_{\text{\rm opt}}(E,\F)$ which are $\P$--orthogonal to $X$ (that is the product $XM$ is an $(\F,\P)$--martingale), with $\P\as$ c\`adl\`ag paths, $M_0=0$ and $\|M\|^p_{\M^p}:=\E^\P\Big[ [ M]^{\frac{p}{2}}_T\Big]<\infty$;

\item $\Lc^{p,2}(E)$ denotes the space of families $(\xi(s))_{s\in [0,T]}$ of $\Fc$--measurable $E$--valued random variables such that the mapping $([0,T]\times \Omega, \Bc([0,T])\otimes \Fc^X_T)\longrightarrow (\Lc^p(E),\| \cdot \|_{\Lc^p}):s\longmapsto \xi(s)$ is continuous, $   \| \xi\|_{\Lc^{p,2}}^p:= \sup_{s\in[0,T]} \|\xi\|^p_{\Lc^{p}}<\infty$;

\item $\Pc^2_{\text{meas}}(E,\Fc)$ of two parameter processes $(U_\uptau)_{\uptau \in [0,T]^2 }$ $:([0,T]^2\times \Omega, \Bc([0,T]^2)\otimes \Gc)  \longrightarrow (\Bc(E),E)$ measurable.\medskip

Finally, given an arbitrary integrability space $(\I^p(E),\|\cdot\|_{\I})$, we introduce the space

\item $\I^{p,2}(E)$ of $(U_\uptau)_{\uptau \in [0,T]^2 }\in \Pc^2_{\text{meas}}(E,\Fc_T)$ such that the mapping $([0,T],\Bc([0,T])) \longrightarrow (\I^{p}(E),\|\cdot \|_{ \I^{p}}): s \longmapsto U^s $ is continuous and $\| U\|_{\I^{p,2}}^p:= \sup_{s\in[0,T]} \| U^s\|_{\I^{p}}^p <\infty$.

For example, $\H^{p,2}(E)$ denotes the space of $(Z_\uptau)_{\uptau \in [0,T]^2 }\in \Pc^2_{\text{meas}}(E,\Fc_T)$ such that the mapping $([0,T],\Bc([0,T])) \longrightarrow (\H^{p}(E),\|\cdot \|_{ \H^{p}}): s \longmapsto Z^s $ is continuous and $\| Z\|_{\H^{p,2}}^p:= \sup_{s\in[0,T]} \| Z^s\|_{\H^{p}}^p <\infty $.

\item $\overline{\H}^{_{\raisebox{-1pt}{$ \scriptstyle p,2$}}}(E)$ of $(Z_\uptau)_{\uptau \in [0,T]^2 }\in \Pc^2_{\text{meas}}(E,\Fc_T)$ such that $Z\in \H^{p,2}(E)$, $\Zc\in \H^2(E)$ where $\Zc:=(Z_t^t)_{t\in [0,T]}$ and 
\[ \|Z\|_{\overline{\H}^{_{\raisebox{-2pt}{$ \scriptstyle p,2$}}}}^2:=\|Z\|_{\H^{2,2}}^2+\|\Zc\|_{\H^2}^2<\infty \]

\end{list}

\begin{remark}
When $p=q$, we will write $\L^{p}(E)$ $\big($resp. $\L^{p,2}(E)\big)$ for $\L^{q,p}(E)$ $\big($resp. $\L^{q,p,2}(E)\big)$. With this convention, $\L^{2}(E)$ $\big($resp. $\L^{2,2}(E)\big)$ will be $\L^{2,2}(E)$ $\big($resp. $\L^{2,2,2}(E)\big)$. Also, $\S^{p,2}(E)$, $\L^{q,p,2}(E)$ and $\H^{p,2}(E)$ are Banach spaces. In addition, we remark that the space $\overline{\H}^{_{\raisebox{-1pt}{$ \scriptstyle p,2$}}}(E)$ is the largest subspace of $\H^{p,2}(E)$ for which the diagonal process $\Zc$ is in $\H^2(E)$. It is also clear that $\overline{\H}^{_{\raisebox{-1pt}{$ \scriptstyle p,2$}}}(E)$ is a Banach space.
\end{remark}

\section{An infinite dimensional system of BSDEs}\label{Section:infdimsystem}

We are given jointly measurable mappings $h$, $g$, $\xi$ and $\eta$, such that for any $(y,z,u,v,{\rm u})\in \R^{d_1\!}\times\R^{n\times d_1 \!}\times\R^{d_2\!}\times\R^{n\times d_2 \!}\times\R^{d_2} $
\begin{align*} 
\begin{split}
 &h:  [0,T] \times \Xc\times \R^{d_1}\! \times \R^{ n\times d_1}\! \times \R^{d_2}\!\times \R^{ n\times d_2}\! \times \R^{d_2}\! \longrightarrow\R^{d_1} ,\;   h_\cdot(\cdot ,y,z,u,v,{\rm u} )\in \Pc_{{\rm prog}}(\R^{d_1},\F),\\
&{g}: [0,T]^2 \times \Xc\times \R^{d_2}\! \times \R^{n\times d_2 }\! \times \R^{d_1}\!\times \R^{n\times d_1}\! \longrightarrow \R^{d_2} ,\;   g_\cdot(s,\cdot ,u,v,y,z)\in \Pc_{{\rm prog}}(\R^{d_2},\F),\\
 &\xi:  [0,T]\times \Xc\longrightarrow \R^{d_1}, \; \eta:[0,T]\times \Xc\longrightarrow \R^{d_2}.
 \end{split}
\end{align*}
Moreover, we work under the following set of assumptions.

\begin{assumption}\label{AssumptionA}

\begin{enumerate}[label=$(\roman*)$, ref=.$(\roman*)$,wide, labelwidth=!, labelindent=0pt]

\item  \label{AssumptionA:i} $(s,u,v) \longmapsto  g_t(s,x,u,v,y,z)$ $($resp. $s\longmapsto \eta(s,x))$ is continuously differentiable, uniformly in $(t,x,y,z)$ $($resp. in $x)$. Moreover, the mapping $\nabla g:[0,T]^2\times \Xc \times (\R^{d_2}\!\times \R^{n\times d_2 } )^2 \! \times \R^{d_1}\!\times \R^{n\times d_1}\! \longrightarrow \R^{d_2}$ defined by \vspace{-1em}

\[ \nabla g_t (s,x,{\rm u},{\rm v},u, v ,y,z):=\partial_s g_t(s,x,u,v,y,z)+\partial_u g_t(s,x,u,v,y,z){\rm u}+\sum_{i=1}^n \partial_{v_{:i}} g_t(s,x,u,v,y,z){\rm v}_{i:},\]\vspace{-1em}

satisfies $\nabla g _\cdot(s,\cdot,{\rm u},{\rm v},u, v ,y,z)\in \Pc_{\rm prog}(\R^{d_2},\F);$

\item \label{AssumptionA:ii}$(y,z,u,v,{\rm u} )\longmapsto  h_t(x,y,z,u,v,{\rm u})$ is uniformly Lipschitz continuous, i.e. $\exists L_h>0,$ such that for all $(t,x,y,\tilde y,$ $z,\tilde z,u,\tilde u,v,\tilde v,{\rm u}, \tilde {\rm u})$
\begin{align*}
 |h_t(x,y,z,u,v,{\rm u})-h_t(x,\tilde y,\tilde z,\tilde u,\tilde v,\tilde {\rm u})|\leq L_h\big(|y-\tilde y|+ |\sigma_t(x)^\t(z-\tilde z)|+|u-\tilde u|+ |\sigma_t(x)^\t(v-\tilde v)|+|{\rm u}-\tilde {\rm u}|\big);
\end{align*}

\item\label{AssumptionA:iii}  for $\varphi \in \{g, \partial_s g \}$, $(u,v,y,z)\longmapsto  \varphi_t(s,x,u,v,y,z)$ is uniformly Lipschitz continuous, i.e. $\exists L_{\varphi } > 0,$ such that for all $(s,t,x,u,\tilde u,v,\tilde v,y,\tilde y,z,\tilde z) $ 
\begin{align*}
 \ |\varphi_t(s,x,u,v,y,z)-\varphi_t(s,x,\tilde u,\tilde v,\tilde y,\tilde z)|\leq L_{\varphi }\big(|u-\tilde u|+|\sigma_t(x)^\t(v-\tilde v')|+|y-\tilde y|+|\sigma_t(x)^\t(z-\tilde z)|\big);
\end{align*}
\item \label{AssumptionA:iv} for ${\bf 0}:=(u,v,y,z)|_{(0,...,0)}$, $\big( \tilde h_\cdot ,  \tilde g_\cdot(s),   \nabla  \tilde g_\cdot(s)\big) :=\big( h_\cdot(\cdot,{\bf 0},0), g_\cdot(s,\cdot,{\bf 0}), \partial _s g_\cdot(s,\cdot,{\bf 0})\big) \in \L^{1,2}(\R^{d_1\!})\times \big(\L^{1,2,2}(\R^{d_2\!})\big)^2$. \label{AssumptionA:gen0}
\end{enumerate}
\end{assumption}
\begin{remark}\label{Remark:assumptionBSDE}
We comment on the set of requirements in {\rm \Cref{AssumptionA}}. Of particular interest is {\rm \Cref{AssumptionA}\ref{AssumptionA:i}}, the other being the standard Lipschitz assumptions on the generators as well as their integrability at zero. Anticipating the introduction of \eqref{Eq:systemBSDE} below and the discussion in {\rm \Cref{Remark:wpsystem}}, {\rm \Cref{AssumptionA}\ref{AssumptionA:i}} will allow us to identify the second {\rm BSDE} in the system as the antiderivative of the third one, see {\rm \Cref{Remark:wpsystem}}.
\end{remark}

Let us define the space $(\Hc,\|\cdot \|_{\Hc})$, whose generic elements we denote $\mathfrak{h}=(\Yc,\Zc,\Nc,U,V,M,\partial U, \partial V, \partial M)$, where
\begin{gather*}
 \Hc :=\S^2(\R^{d_1}\!) \times \H^2(\R^{n\timesr d_1}\!)\times  {\M^2}(\R^{d_1}\!)\times \S^{2,2}(\R^{d_2}\!) \times \Ho(\R^{n\timesr d_2 }\!) \times {\M^{2,2}}(\R^{d_2}\!)\times \S^{2,2}(\R^{d_2}\!) \times \H^{2,2}(\R^{n\timesr d_2 }\!) \times {\M^{2,2}}(\R^{d_2}\!),\\[0,3em]
 \|\mathfrak{h} \|^2_{\Hc} := \|\Yc\|_{\S^2}^2 + \|\Zc\|_{\H^2}^2+\|\Nc\|_{\M^2}^2+ \|U\|_{\S^{2,2}}^2 + \|V\|_{ \Ho}^2+\|M \|_{\M^{2,2}}^2+ \|\partial U\|_{\S^{2,2}}^2 + \|\partial V\|_{\H^{2,2}}^2+\|\partial M \|_{\M^{2,2}}^2,
\end{gather*}

We are now ready to precise the class of systems subject to our study. Given $(\xi,  \eta, \partial_s \eta )\in \Lc^2(\R^{d_1}\!)\times ( \Lc^{2,2} (\R^{d_2}\!))^2$, we consider the system, supposed to hold $\P\as$
\begin{align*}\label{Eq:systemBSDE}\tag{$\Sc$}
\begin{split}
\Yc_t&=\xi(T,X_{\cdot\wedge T})+\int_t^T h_r(X,\Yc_r,\Zc_r, U_r^r,V_r^r,\partial U_r^r )\d r-\int_t^T  \Zc_r^\t \d X_r-\int_t^T \d \Nc_r,\; t \in[0,T], \\
U_t^s&=   \eta (s,X_{\cdot\wedge T})+\int_t^T  g_r(s,X,U_r^s,V_r^s, \Yc_r, \Zc_r) \d r-\int_t^T {V_r^s}^\t  \d X_r-\int_t^T \d M^s_r,\; (s,t)\in[0,T]^2,  \\
\partial U_t^s&=  \partial_s \eta (s,X_{\cdot\wedge T})+\int_t^T  \nabla g_r(s,X,\partial U_r^s,\partial V_r^s,U_r^s,V_r^s, \Yc_r, \Zc_r) \d r-\int_t^T\partial  {V_r^s}^\t  \d X_r-\int_t^T \d \partial M^s_r,\; (s,t)\in[0,T]^2.
\end{split}
\end{align*}

\begin{definition}\label{Def:solsystem}
We say $\mathfrak{h}$ is a solution to \eqref{Eq:systemBSDE} if $\mathfrak{h}\in \Hc$ and \eqref{Eq:systemBSDE} holds. \end{definition}

\begin{remark}\label{Remark:wpsystem}
\begin{enumerate}[label=$(\roman*)$, ref=.$(\roman*)$,wide, labelwidth=!, labelindent=0pt]
We now expound on our choice for the set--up and the structure of \eqref{Eq:systemBSDE}.

\item We first highlight two aspects which are crucial to establish the connection between \eqref{Eq:systemBSDE} and {\rm type--I BSVIE} \eqref{Eq:typeIBSVIEfe}. The first is the presence of $\partial U$ in the generator of the first equation. This causes the system to be fully coupled but is nevertheless necessary in our methodology, this will be clear from the proof of {\rm \Cref{Thm:wpbsvie}} in {\rm \Cref{Section:BSVIE}}. The second relates to our choice to write three equations instead of two. In fact, our approach is based on being able to identify $\partial U$ as the derivative with respect to the $s$ variable of $U$ in an appropriate sense and, at least formally, it is clear that the third equation allows us to do so, see {\rm \Cref{Lemma:partialU}} for details. Alternatively, we could have chosen not to write the third equation and consider
\begin{align*}
\Yc_t&=\xi(T,X_{\cdot\wedge T})+\int_t^T h_r(X,\Yc_r,\Zc_r, U_r^r,V_r^r,\partial U_r^r )\d r-\int_t^T  \Zc_r^\t \d X_r-\int_t^T \d \Nc_r,\; t \in[0,T], \\
U_t^s&=   \eta (s,X_{\cdot\wedge T})+\int_t^T  g_r(s,X,U_r^s,V_r^s, \Yc_r, \Zc_r) \d r-\int_t^T {V_r^s}^\t  \d X_r-\int_t^T \d M^s_r,\; (s,t)\in[0,T]^2,  \\
\partial U^s_t& :=\frac{\d }{\d s} U^s_{|_{(s,t)}},\; (s,t)\in[0,T]^2,
\end{align*}
 
where $\frac{\d }{\d s} U^s$ corresponds to the density with respect to the Lebesgue measure of $s\longmapsto U^s$. Nevertheless, for the proof of well--posedness of \eqref{Eq:systemBSDE} that we present in {\rm \Cref{Section:Analysis}}, we have to derive appropriate estimates for $(\partial U_t^t)_{t\in [0,T]}$, and for this it is advantageous to do the identification by adding the third equation in \eqref{Eq:systemBSDE} and work on the space $(\Hc,\|\cdot \|_{\Hc})$.

\item \label{Remark:wpsystem:ii} We also emphasise that the presence of $(V_t^t)_{t\in[0,T]}$ in the generator of the first equation require us to reduce the space of the solution from the classic $(\Hf, \|\cdot \|_{\Hf})$ to $(\Hc,\|\cdot \|_{\Hc})$ where
\[\Hf :=\S^2(\R^{d_1}\!) \times \H^2(\R^{n\timesr d_1}\!)\times  {\M^2}(\R^{d_1}\!)\times \S^{2,2}(\R^{d_2}\!) \times \H^{2,2}(\R^{n\timesr d_2 }\!) \times {\M^{2,2}}(\R^{d_2}\!)\times \S^{2,2}(\R^{d_2}\!) \times \H^{2,2}(\R^{n\timesr d_2 }\!) \times {\M^{2,2}}(\R^{d_2}\!),\]
and $\|\cdot\|_{\Hf}$ denotes the norm induced by $\Hf$. Ultimately, this is due to the presence of $(Z_t^t)_{t\in [0,T]}$ in the {\rm type--I BSVIE} \eqref{Eq:typeIBSVIEfe}. On this matter, we stress that to the best of our knowledge, our results constitute the first comprehensive study of {\rm type--I BSVIEs} as general as \eqref{Eq:typeIBSVIEfe}. As such, part of our contributions is the identification of the appropriate space to carry out the analysis. In the case where $(V_t^t)_{t\in[0,T]}$ $($resp. $(Z_t^t)_{t\in [0,T]}\big)$ does not appear in the generator of the first {\rm BSDE} in \eqref{Eq:systemBSDE} $($resp. {\rm type--I BSVIE} \eqref{Eq:typeIBSVIEfe}$)$, {\rm \Cref{Prop:aprioriest:simplify}} and {\rm \Cref{Remark:contractionsimplify}} $($resp. {\rm \Cref{Remark:wpBSVIE}\ref{Remark:wpBSVIE:ii}}$)$ provide the arguments on how one can adapt our approach to yield a solution in the classical space. This shows that our methodology recovers existing results on {\rm type--I BSVIE} \eqref{Eq:typeIBSVIE} as well as the so--called {\rm extended type--I BSVIE} \eqref{Eq:typeIBSVIEextended}.

\item \label{Remark:wpsystem:iii} Having justified our choice of set--up for \eqref{Eq:systemBSDE}, we also note that one could have alternatively considered the space\footnote{$(\Hf^\Dc, \|\cdot\|_{\Hf})$ is a Banach space. Indeed, the space $(\Hf,\|\cdot \|_{\Hf})$ is clearly a Banach space, and the dominated convergence theorem guarantees that the space remains closed under the addition of \eqref{Eq:constraintspace}.} $(\Hf^\Dc, \|\cdot\|_{\Hf})$, of $\mathfrak{h}\in \Hf$ with $\|\mathfrak{h}\|_{\Hf}<\infty$,
such that
\begin{align*}\label{Eq:constraintspace}\tag{$\Dc$}
\bigg(\int_s^T \partial U^r \d r,\int_s^T \partial V^r \d r,\int_s^T \partial M^r \d r\bigg)=\big(U^T-U^s,V^T-V^s,M^T-M^s \big),\; \forall s\in [0,T],
\end{align*}
in $\S^2(\R^{d_2}\!) \times \H^2(\R^{n\timesr d_2}\!)\times  {\M^2}(\R^{d_2}\!)$. This is similar to the discussion in {\rm \cite[Section 2.1]{hamaguchi2020extended}}. The advantage of introducing \eqref{Eq:constraintspace} is that it ensures the existence of $(V_t^t)_{t\in [0,T]}\in \H^2(\R^{n\timesr d_2}\!)$, see {\rm \Cref{Lemma:partialU}}, implying $\Hc^\Dc\subseteq \Hc$. Consequently, this choice would reduce the generality of our well--posedness result. Fortunately, as a by--product of our approach, the solution to \eqref{Eq:systemBSDE} in $(\Hc,\|\cdot \|_{\Hc})$ happens to be in $(\Hf^\Dc,\|\cdot \|_{\Hf})$, see {\rm \Cref{Thm:solsatD}}.

\end{enumerate}
\end{remark}

\begin{remark}
In addition, we highlight two features of \eqref{Eq:systemBSDE} that will come into play in the setting of {\rm type--I BSVIE} \eqref{Eq:typeIBSVIEfe}, and differ from the one in the classic literature. They are related to the fact we work under the general filtration $\F$. The first is the fact that the stochastic integrals in \eqref{Eq:systemBSDE} are with respect to the canonical process $X$. Recall that $\sigma$ is not assumed to be invertible $($it is not even a square matrix in general and can vanish$)$, therefore the filtration generated by $X$ is different from the one generated by $B$. This yields more general results and it allows for extra flexibility necessary in some applications, see {\rm \cite{hernandez2020me}} for an example. The second difference is the presence of the processes $(N,M,\partial M)$. As it was mentioned in {\rm \Cref{Section:stochasticbasis}}, we work with a probability measure for which the martingale representation property for $\F$--local martingales in terms of stochastic integrals with respect to $X$ does not necessarily hold. Therefore, we need to allow for orthogonal martingales in the representation. Certainly, there are known properties which are equivalent to the orthogonal martingales vanishing, i.e. $N=M=\partial M=0$, for example when $\P$ is an extremal point of the convex hull of the probability measures that satisfy the properties in {\rm \Cref{Section:stochasticbasis}}, see {\rm \cite[Theorem 4.29]{jacod2003limit}}.
\end{remark}

\Cref{AssumptionA} provides an appropriate framework to derive the well--posedness of \eqref{Eq:systemBSDE}. The following is the main theorem of this section whose proof we postpone to \Cref{Section:Analysis}.

\begin{theorem}\label{Thm:wp}
Let {\rm \Cref{AssumptionA}} hold. Then \eqref{Eq:systemBSDE} admits a unique solution in $(\Hc,\|\cdot\|_{\Hc})$. For any $\mathfrak{h}\in \Hc$ solution to \eqref{Eq:systemBSDE} there exists $C>0$ such that
\begin{align*}
\|(\Yc,\Zc,\Nc,U, V,M,\partial U,\partial V,\partial M) \|^2_{{\Hc}} \leq C \Big(  \|\xi \|^2_{\Lc^2}  + \|\eta \|^2_{\Lc^{2,2}}+\| \partial_s \eta \|^2_{\Lc^{2,2}} +\|\tilde h\|_{\L^{1,2}}^2 +\| \tilde g \|_{\L^{1,2,2}}^2+\| \nabla \tilde g \|_{\L^{1,2,2}}^2 \Big).
\end{align*}
Moreover, if $\mathfrak{h}^i \in \Hc$ denotes the solution to \eqref{Eq:systemBSDE} with coefficients $(\xi^i,h^i,\eta^i, g^i,\partial_s \eta^i, \nabla g^i)$ for $i\in\{1,2\}$, then
\begin{align*}
\| \delta \mathfrak{h}  \|^2_{\Hc} \leq &\ C\Big(  \|\delta \xi\big\|^2_{\Lc^2} +\|\delta \eta \big\|^2_{\Lc^{2,2}}+\|\delta \partial_s \eta \big\|^2_{\Lc^{2,2}} +  \|\delta_1 h \|_{\L^{1,2}}^2+ \|\delta_1  g \|_{\L^{1,2,2}}^2+\|\delta_1  \nabla g \|_{\L^{1,2,2}} \Big),  
\end{align*}
where for $\varphi\in\{\Yc,\Zc,\Nc,U,V,M,\partial U,\partial V,\partial M,\xi,\eta,\partial_s \eta \}$ and $\Phi\in \{h,g,\nabla g\}$
\[
\delta \varphi:= \varphi^1-\varphi^2,\; \text{\rm and}\; \delta_1 \Phi_t:= \Phi^1_t( \Yc_r^1, \Zc^1_t,U^{1t}_t,V^{1t}_t)- \Phi^2_t( Y_r^1, Z^1_t,U^{1t}_t,V^{1t}_t),\; \d t\otimes \d \P \ae \text{ \rm on }[0,T]\times \Xc.
\]
\end{theorem} 

The next result proves a fortuitous consequence of our set-up to study \eqref{Eq:systemBSDE}. In words, it says that although we carry out the study of \eqref{Eq:systemBSDE} under the general space $(\Hc,\|\cdot\|_{\Hc})$, the unique solution established by \Cref{Thm:wp} is in the space $(\Hf^\Dc,\|\cdot\|_{\Hf})$, see again {\rm \cite[Section 2.1]{hamaguchi2020extended}}.
\begin{proposition}\label{Thm:solsatD}
Let {\rm \Cref{AssumptionA}} hold, and let $\mathfrak{h}$ be the solution to \eqref{Eq:systemBSDE} in $(\Hc,\|\cdot\|_{\Hc})$. Then ${\mathfrak h} \in \Hf^\Dc$.
\begin{proof}
The result follows from \Cref{Lemma:partialU}.
\end{proof}
\end{proposition}

\begin{remark}\label{Rmk:quadratic}
The reader may wonder about our choice to leave out the diagonal of $\partial V$ in the generator of the first equation in \eqref{Eq:systemBSDE}. As we will argue below, this would require us to consider an auxiliary infinite dimensional family of quadratic {\rm BSDEs}. Since the main purpose of this paper is to relate the well--posedness of \eqref{Eq:systemBSDE} to that of the {\rm type--I BSVIE} \eqref{Eq:typeIBSVIEfe}, and inasmuch as we do not need to consider this case to establish {\rm \Cref{Thm:wpbsvie}}, we have refrained to pursue it in this document. Nevertheless, this case is covered as part of the study of the extension of \eqref{Eq:systemBSDE} to the quadratic case in {\rm \citet*{hernandez2020infinite}}. If we were to study the system
\begin{align*}
\begin{split}
\Yc_t&=\xi(T,X_{\cdot\wedge T})+\int_t^T h_r(X,\Yc_r,\Zc_r, U_r^r,V_r^r,\partial  U_r^r ,\partial V_r^r)\d r-\int_t^T  \Zc_r^\t  \d X_r-\int_t^T \d N_r,\; t\in[0,T],\\
U_t^s&=   \eta (s,X_{\cdot\wedge T})+\int_t^T  g_r(s,X,U_r^s,V_r^s, \Yc_r, \Zc_r) \d r-\int_t^T {V_r^s}^\t \d X_r-\int_t^T \d M^s_r,\; (s,t)\in[0,T]^2,\\
\partial U_t^s&=  \partial_s \eta (s,X_{\cdot\wedge T})+\int_t^T  \nabla g_r(s,X,\partial U_r^s,\partial V_r^s,U_r^s,V_r^s, \Yc_r, \Zc_r) \d r-\int_t^T\partial  {V_r^s}^\t   \d X_r-\int_t^T \d \partial M^s_r,\; (s,t)\in[0,T]^2,
\end{split}
\end{align*}

and as it is clear from our analysis in {\rm \Cref{Section:Analysis}}, its well--posedness requires both having a rigorous method to define the mapping $t\longmapsto \partial V_t^t$, as well as deriving a priori estimates for the norm of $\partial V_t^t$. In analogy with {\rm \Cref{Lemma:partialU}} and {\rm \Cref{Remark:wpsystem}}, both tasks require us to make sense of the family of {\rm BSDEs} with terminal condition $\partial_{ss}\eta$ and generator 
\begin{align*}
\nabla^2 g_t(s,x, {\rm \tilde u}, \tilde {\rm v}, {\rm u}, {\rm v},u, v ,y,z):= \nabla g_t(s,x, {\rm \tilde u}, \tilde {\rm v},u, v ,y,z) + \sum_{(\pi_i,\pi_j,\tilde \pi_i,\tilde \pi_j)\in \Pi^2\times \widetilde \Pi^2}  \tilde \pi_i^\t\partial_{\pi_i  \pi_j }^2 g_t(s,x,u, v ,y,z) \tilde \pi_j  ,
\end{align*}\vspace{-1em}

where $\Pi:=\big(s, u,v_{1:},..., v_{n:} \big)$, $\tilde \Pi:=\big(1,{\rm u},{\rm v}_{1:},...,{\rm v}_{n:} \big)$ and $ \partial_{\pi_i  \pi_j }^2 g_t(s,x,u, v ,y,z)$ denote the second order derivatives of $g$. Even though we could add assumptions ensuring that the second order derivatives are bounded, it is clear from the second term in the generator that we would necessarily need to consider a quadratic framework.
\end{remark}

\section{Well--posedness of type--I BSVIEs}\label{Section:BSVIE}

We now address the well--posedness of type--I BSVIEs. Let $d$ be a non--negative integer, and $f$ and $\xi$ be jointly measurable functionals such that for any $(s,y,z,u,v)\in [0,T]\times (\R^d \times \R^{n\times d})^2$
\begin{align*}
 & f: [0,T]^2\times \Xc\times (\R^d \times \R^{n\times d})^2 \longrightarrow \R^d ,\; f_\cdot(s,\cdot ,y,z,u,v )\in \Pc_{{\rm prog}}(\R^{d},\F),\\
& \xi: [0,T]\times \Xc\longrightarrow \R^d,\; \xi(s,\cdot) \; \text{\rm is}\; \Fc\text{\rm --measurable}.
\end{align*}

To derive the main result in this section, we will exploit the well--posedness of the infinite dimensional system of BSDE considered in \Cref{Section:infdimsystem}. Therefore, we work under the following set of assumptions.

\begin{assumption}\label{Assumption:SystemBSVIEwp}
\begin{enumerate}[label=$(\roman*)$, ref=.$(\roman*)$,wide, labelwidth=!, labelindent=0pt]
\item \label{Assumption:SystemBSVIEwp:i}$(s,y,z)\longmapsto  f_t(s,x,y,z,u,v)$ $($resp. $s\longmapsto \xi(s,x))$ is continuously differentiable, uniformly in $(t,x,u,v)$ $($resp. in $x)$. Moreover, the mapping $\nabla f:[0,T]^2\times \Xc \times (\R^{d}\times \R^{n\times d } )^3\longrightarrow \R^{d}$ defined by \vspace{-1em}

\[\nabla f_t(s,x,{\rm u}, {\rm v}, y,z ,u,v):=\partial_s f_t(s,x,y,z,u,v)+\partial_y f_t(s,x,y,z,u,v){\rm u}+\sum_{i=1}^n \partial_{z_{:i}} f_t(s,x,y,z,u,v){\rm v}_{i:},
\] \vspace{-1em}

satisfies $\nabla f_\cdot(s,\cdot,y,z,u,v,{\rm u}, {\rm v})\in \Pc_{\rm prog}(\R^d,\F)$ for all $s\in [0,T];$

\item \label{Assumption:SystemBSVIEwp:ii}$(y,z,u,v)\longmapsto  f_t(t,x,y,z,u,v)$ is uniformly Lipschitz continuous, i.e. $\exists L_h>0,$ such that for all $(t,x,y,\tilde y,z,\tilde z,$ $u,\tilde u,v,\tilde v)$
\begin{align*}
 |f_t(t,x,y,z,u,v)-f_t(t,x,\tilde y,\tilde z,\tilde u,\tilde v)|\leq L_h\big(|y-\tilde y|+ |\sigma_t(x)^\t(z-\tilde z)|+|u-\tilde u|+ |\sigma_t(x)^\t(v-\tilde v)|\big);
\end{align*}

\item \label{Assumption:SystemBSVIEwp:iii} for $\varphi \in \{f, \partial_s f\}$, $(u,v,y,z)\longmapsto  \varphi_t(s,x,y,z,u,v)$ is uniformly Lipschitz continuous, i.e. $\exists L_{\varphi } > 0,$ such that for all $(s,t,x,y,\tilde y,z,\tilde z,u,\tilde u,v,\tilde v)$ 
\begin{align*}
 \ |\varphi_t(s,x,y,z,u,v)-\varphi_t(s,x,\tilde y,\tilde z,\tilde u,\tilde v)|\leq L_{\varphi }\big(|y-\tilde y|+ |\sigma_t(x)^\t(z-\tilde z)|+|u-\tilde u|+ |\sigma_t(x)^\t(v-\tilde v)|\big).
\end{align*}

\item \label{Assumption:SystemBSVIEwp:iv} $\big( \tilde f_\cdot ,\tilde f_\cdot(s) ,  \nabla\tilde f_\cdot(s)\big) :=\big( f_\cdot(\cdot,\cdot ,{\bf 0}),f_\cdot(s,\cdot ,{\bf 0}), \partial_s f_\cdot(s,\cdot,{\bf 0})\big)   \in \L^{1,2}(\R^{d})\times \big(\L^{1,2,2}(\R^{d})\big)^2$.
\end{enumerate}
\end{assumption}

Let $(\Hc^\star, \|\cdot \|_{\Hc^\star})$ denote the space of $(Y,Z,N)\in \Hc^\star$ such that $\|(Y,Z,N)\|_{\Hc^\star}<\infty$ where
\begin{gather*}
 \Hc^\star:= \S^{2,2}(\R^d)\times \Ho(\R^{n\times d} ) \times \M^{2,2}(\R^d),\;
 \|\cdot\|_{\Hc^\star}:=\|Y\|_{\S^{2,2}}^2 + \|Z\|_{\Ho}^2+\|N \|_{\M^{2,2}}^2.
\end{gather*}

We consider the $n$--dimensional type--I BSVIE on $(\Hc^\star,\|\cdot\|_{\Hc^\star})$ given by
\begin{align}\label{Eq:bsvie}
Y_t^s =   \xi (s,X)+\int_t^T  f_r(s,X,Y_r^s,Z_r^s, Y_r^r, Z_r^r) \d r-\int_t^T {Z_r^s}^\t\d X_r-\int_t^T \d N^s_r,\; (s,t)\in[0,T]^2,\; \P\as
\end{align}

We work under the following notion of solution.

\begin{definition}\label{Def:soltypeIBSVIEfe}
We say $(Y,Z,N)$ is a solution to the {\rm type--I BSVIE} \eqref{Eq:bsvie} if $(Y,Z,N)\in \Hc^\star$ verifies \eqref{Eq:bsvie}.
\end{definition}

Defining $h_t(x,y,z,u,v, {\rm u} ):=f_t(t,x,y,z,u,v)-{\rm u}$, we may consider the system, supposed to hold $\P\as$
\begin{align}\label{Eq:systemBSDEf}\tag{$\Sc_f$}
\begin{split}
\Yc_t&=\xi(T,X)+\int_t^T h_r(X,\Yc_r,\Zc_r,Y_r^r,Z_r^r,\partial Y_r^r)\d r-\int_t^T \Zc_r^\t \d X_r-\int_t^T \d \Nc_r,\; t\in [0,T], \\
Y_t^s&=   \xi (s,X)+\int_t^T  f_r(s,X,Y_r^s,Z_r^s, \Yc_r, \Zc_r) \d r-\int_t^T {Z_r^s}^\t  \d X_r-\int_t^T \d N^s_r,\; (s,t)\in[0,T]^2,\\
\partial Y_t^s&=  \partial_s \xi (s,X)+\int_t^T  \nabla f_r(s,X, \partial Y_r^s,\partial Z_r^s ,Y_r^s,Z_r^s, \Yc_r, \Zc_r) \d r-\int_t^T\partial  {Z_r^s}^\t  \d X_r-\int_t^T \d \partial N^s_r,\; (s,t)\in[0,T]^2.
\end{split}
\end{align}

\begin{remark}
We now make a few comments regarding our set--up for the study {\rm type--I BSVIE} \eqref{Eq:bsvie}.
\begin{enumerate}[label=$(\roman*)$, ref=.$(\roman*)$,wide, labelwidth=!, labelindent=0pt]

\item The necessity of the set of assumptions in {\rm \Cref{Assumption:SystemBSVIEwp}} to our approach, based on the systems introduced in {\rm \Cref{Section:infdimsystem}}, is clear. Compared to the set of assumption made by recent works on {\rm BSVIEs} in the literature we notice the main difference is the regularity with respect to the $s$ variable we imposed on the data of the problem, i.e. {\rm \Cref{Assumption:SystemBSVIEwp}\ref{Assumption:SystemBSVIEwp:i}}. In particular, we highlight that {\rm type--I BSVIE} \eqref{Eq:typeIBSVIEextended}, in which the diagonal of $Y$, but not of $Z$ is allowed in the generator, had been considered in {\rm \cite{hamaguchi2020extended, wang2020extended}}. In such a scenario, the authors assumed $(\xi,f)\in \Lc^{2,2}(\R^d)\times \L^{1,2,2}(\R^d)$, and no additional condition is required to obtain the well--posedness of \eqref{Eq:typeIBSVIEextended}. As it will be clear from {\rm \Cref{Prop:aprioriest:simplify}} and {\rm \Cref{Remark:wpBSVIE}} our procedure can be adapted to work under such set of assumptions provided the diagonal of $Z$ is not considered in the generator.

\item Moreover, the spaces of the solution considered in {\rm \cite{hamaguchi2020extended, wang2020extended}} also differ, echoing the absence of the diagonal of $Z$ in the generator. The authors work with the notion of {\rm C--solution}, that is, $Y$ is assumed to be a jointly measurable process, such that $s\longmapsto Y^s$ is continuous in $\L^{1,p}(\R^d)$, $p\geq 2$, and for every $s\in [0,T]$, $Y^s$ is $\F$--adapted with $\P\as$ continuous paths. This coincides with our definition of the space $\L^{1,p,2}(\R^d)$. Similarly, $Z$ belongs to the space $\H^{2,2}(\R^{n\times d})$. On the other hand, {\rm \cite{wang2019time}} provides a representation formula for {\rm type--I BSVIEs} for which the driver allows for the diagonal of $Z$, but not of $Y$. More precisely, they introduce a {\rm PDE}, similar to the one we will introduce in {\rm \Cref{Section:PDEsandtimeinconsistency}}, prove its well--posedness, and then a Feynman--Kac formula. Naturally, in this case $(Y,Z)$ inherits the regularity of the underlying {\rm PDE}.

\item The main contributions of our methodology to the field of {\rm BSVIEs} is to be able to accommodate {\rm type--I BSVIEs} for which the diagonal of $Z$ appears in the generator. For this, the definition of the space $(\Hc^\star, \|\cdot \|_{\Hc^\star})$ and {\rm \Cref{Assumption:SystemBSVIEwp}\ref{Assumption:SystemBSVIEwp:i}} play a central role. As first noticed in {\rm \cite{hernandez2020me}}, under this assumption one can identify, see {\rm \Cref{Prop:aprioriestimatesdif}}, the unique process $\partial Z^s$ which can be understood as the derivative of $s\longmapsto Z^s$. We recall again this is in spirit of the recent argument in {\rm \cite[Section 2.1]{hamaguchi2020extended}} where assuming that $s\longmapsto Z^s$ is continuously differentiable, the author is able to argue the existence of a unique process $(Z_t^t)_{t\in[0,T]}$. We highlight that our approach has the additional advantage of being able to identify the dynamics of $(\partial Y,\partial Z)$. Moreover, {\rm \Cref{Assumption:SystemBSVIEwp}\ref{Assumption:SystemBSVIEwp:i}}, being an assumption on the data of the {\rm BSVIE}, is much easier to verify in practice as opposed to the regularity required in {\rm \cite{hamaguchi2020extended}}. Certainly, our results would still hold true if we require the differentiability of data $(\xi,f)$ with respect to the parameter $s$ in the $\Lc^2$, respectively $\L^{1,2}$, sense.

\item Lastly, we stress that the above {\rm type--I BSVIE} is defined for $(s,t)\in [0,T]^2$, as opposed to $0\leq s\leq t\leq T$. However, anticipating the result of {\rm \Cref{Thm:wpbsvie}}, this could be handled by first solving on $(s,t)\in [0,T]^2$ and then consider the restriction to $0\leq s\leq t\leq T$.
\end{enumerate}
\end{remark}

We are now in position to prove the main result of this paper. The next result shows that under the previous choice of data for \eqref{Eq:systemBSDEf}, its solution solves the {\rm type--I BSVIE} with data $(\xi,f)$ and {\it vice versa}.

\begin{theorem}\label{Thm:wpbsvie}
Let {\rm \Cref{Assumption:SystemBSVIEwp}} hold. Then
\begin{enumerate}[label=$(\roman*)$, ref=.$(\roman*)$,wide, labelwidth=!, labelindent=0pt]
\item the well--posedness of \eqref{Eq:systemBSDEf} is equivalent to that of the {\rm type--I BSVIE} \eqref{Eq:bsvie}$;$
\item the {\rm type--I BSVIE} \eqref{Eq:bsvie} is well--posed, and or any $(Y,Z,N)\in \Hc^\star$ solution to {\rm type--I BSVIE} \eqref{Eq:bsvie} there exists $C>0$ such that
\begin{align}\label{Eq:aprioriestBSVIE}
\|( Y,Z,N) \|_{\Hc^\star}\leq C \Big(   \|\xi \|^2_{\Lc^{2,2}}+  \|\partial_s \xi \|^2_{\Lc^{2,2}}  +\| \tilde f \|_{\L^{1,2,2}}^2 +\| \nabla \tilde f \|_{\L^{1,2,2}}^2 \Big).
\end{align}
Moreover, if $(Y^i,Z^i,N^i)\in \Hc^\star$ denotes the solution to {\rm type--I BSVIE} \eqref{Eq:bsvie} with data $(\xi^i,f^i)$ for $i\in \{1,2\}$, we have
\begin{align*}
\| (\delta Y,\delta Z,\delta N)  \|^2_{\Hc^\star} \leq &\ C\Big(  \|\delta \xi\big\|^2_{\Lc^2} +\|\delta \partial_s  \xi \big\|^2_{\Lc^{2,2}}+  \|\delta_1 f \|_{\L^{1,2}}^2+ \|\delta_1  \nabla f \|_{\L^{1,2,2}}^2 \Big).
\end{align*}

\end{enumerate}
\end{theorem}
\begin{proof} $(ii)$ is a consequence of $(i)$. Indeed, \eqref{Eq:aprioriestBSVIE} follows from \Cref{Prop:aprioriestimates}, and the well--posedness of {\rm type--I BSVIE} \eqref{Eq:bsvie} from that of \eqref{Eq:systemBSDEf}, which holds by \Cref{Assumption:SystemBSVIEwp} and \Cref{Thm:wp}. We now argue $(i)$.\medskip

Let $(\Yc,\Zc,\Nc,Y,Z,N,\partial Y,\partial Z,\partial N)\in \Hc$ be a solution to \eqref{Eq:systemBSDEf}. It then follows from \Cref{Lemma:intdiagpartialU} that
\begin{align}\label{Eq:claimwpbsvie}
Y_t^t&=\xi(T,X)+\int_t^T h_r(X,Y_r^r,Z_r^r,\Yc_r,\Zc_r,\partial Y_r^r)\d r-\int_t^T {Z_r^r}^\t \d X_r- \int_t^T \d \widetilde N_r, \ t\in [0,T],\; \P\as
\end{align}

where $\widetilde N_t := N_t^t-\int_0^t \partial N_r^r \d r$, $t\in[0,T]$, and $\widetilde N\in \M^2(\R^d)$. This shows that $\big((Y_t^t)_{t\in[0,T]},(Z_t^t)_{t\in[0,T]},\Yc_\cdot,\Zc_\cdot,  (\widetilde N_t)_{t\in [0,T]} \big)$, solves the first BSDE in \eqref{Eq:systemBSDEf}. It then follows from the well--posedness of \eqref{Eq:systemBSDEf}, which holds by \Cref{Assumption:SystemBSVIEwp} and \Cref{Thm:wp}, that $\big((Y_t^t)_{t\in[0,T]},(Z_t^t)_{t\in[0,T]}, (\widetilde N_t)_{t\in [0,T]}\big)=(\Yc_\cdot,\Zc_\cdot, \Nc_\cdot)$ in $\S^2(\R^d)\times \H^2(\R^{n\times d})\times\M^2(\R^d)$. Consequently
\begin{align*}
Y_t^s= \xi(s,X)+\int_t^T f_r(s,X,Y_r^s, Z_r^s, Y_r^r,Z_r^r)\d r-\int_t^T {Z_r^s}^\t \d X_r-\int_t^T \d N_r^s, \; (s,t)\in [0,T]^2,\;  \P\as
\end{align*}  

We are left to show the converse result. Let $(Y,Z,N)\in \Hc^\star$ be a solution to type--I BSVIE \eqref{Eq:bsvie}. We begin by noticing that the processes $\Yc:=(Y_t^t)_{t\in [0,T]},\Zc:=(Z_t^t)_{t\in [0,T]},\Nc:=(N_t^t)_{t\in [0,T]}$ are well--defined. Moreover, $\Zc\in \H^2(\R^{n\times d})$ follows from $Z\in \Ho(\R^{n\times d})$,  and $\Yc\in \L^{2,2}(\R^d)$ follows from
\begin{align*}
\|\Yc\|_{\L^2}^2 =\E\bigg[ \int_0^T |Y_r^r|^2 \d r\bigg]\leq \E \bigg[ \int_0^T \sup_{t\in [0,T]} |Y_t^r|^2\d r\bigg]  =\int_0^T \|Y^r\|_{\S^2} \d r<\infty.
\end{align*}

Then, since \Cref{Assumption:SystemBSVIEwp} holds and $(\Yc,\Zc, Y,Z,N)\in \S^2(\R^d)\times \H^{2}(\R^{n\times d} ) \times \S^{2,2}(\R^d)\times \H^{2,2}(\R^{n\times d} ) \times \M^{2,2}(\R^d)$, we can apply \Cref{Lemma:partialU} and obtain the existence of $(\partial Y, \partial Z, \partial N)\in \S^{2,2}(\R^d)\times \H^{2,2}(\R^{n\times d} ) \times \M^{2,2}(\R^d)$ such that
\[
\partial Y_t^s=  \partial_s \xi (s,X)+\int_t^T  \nabla f_r(s,X, \partial Y_r^s,\partial Z_r^s,Y_r^s,Z_r^s, \Yc_r, \Zc_r) \d r-\int_t^T\partial  {Z_r^s}^\t  \d X_r-\int_t^T \d \partial N^s_r,\; (s,t)\in[0,T]^2,\; \P\as
\]
We claim that $\mathfrak{h}:= (\Yc , \Zc ,  \widetilde N , Y,  Z,  N,  \partial Y, \partial Z, \partial N)$ is a solution to \eqref{Eq:systemBSDEf}, where $\widetilde N_\cdot:=\Nc_\cdot-\int_0^\cdot \partial N_r^r \d r$. For this, we first note that in light of Lemmata \ref{Lemma:partialU} and \ref{Lemma:intdiagpartialU} we have that 
\begin{align}\label{Eq:claimwpbsvie0}
\Yc_t&=\xi(T,X)+\int_t^T h_r(X,\Yc_r,\Zc_r,Y_r^r,Z_r^r,\partial Y_r^r)\d r-\int_t^T {\Zc_r}^\t \d X_r-\int_t^T \d \widetilde N_r , \; t \in [0,T],\; \P\as
\end{align}
and $\widetilde \Nc\in \M^{2,2}(\R^d)$. We are only left to argue $\Yc\in \S^2(\R^d)$. Note that by \Cref{Assumption:SystemBSVIEwp} and \Cref{Eq:ineqsquare} there exists $C>0$ such that
\begin{align*}
 | \Yc_t |^2 &\leq  C\bigg( |\xi(T,X) |^2+\bigg(\int_0^T |\tilde f_r|\d r \bigg)^2  + \int_0^T \big( |\Yc_r|^2+|\sigma_r^\t \Zc_r|^2+  |Y_r^r|^2+|\sigma_r^\t Z_r^r|^2 + |\partial Y_r^r|^2\big) \d r \\
 &\hspace*{2.5em} +\bigg| \int_t^T \Zc_r^\t \d X_r \bigg|^2+\bigg| \int_t^T \d \widetilde N_r\bigg|^2 \bigg), 
\end{align*}
Moreover, by Doob's inequality we have $\displaystyle\E\bigg[ \sup_{t\in [0,T]} \bigg| \int_0^t \Zc_r^\t   \d X_r\bigg|^2\bigg]\leq 4 \| \Zc\|^2_{\H^2}$ and thus {\color{black} \Cref{Eq:ineqpartialUtt2} yields
\begin{align*}
\|\Yc\|_{\S^2}^2\leq C \Big(   \|\xi \|^2_{\Lc^{2,2}}+  \|\partial_s \xi \|^2_{\Lc^{2,2}}  +\| \tilde f \|_{\L^{1,2,2}}^2 +\| \nabla \tilde f \|_{\L^{1,2,2}}^2  +\|\Yc\|_{\L^2}^2+\|Y\|_{\L^{2,2}}^2+ \|Z\|_{\Ho}^2+\|\partial Z\|_{\H^{2,2}}^2\Big) <\infty.
\end{align*}
We conclude $\|\mathfrak{h}\|_{\Hc}<\infty$, $\mathfrak{h}\in \Hc$ and thus $\mathfrak{h}$ solves \eqref{Eq:systemBSDEf}.}
\end{proof}

\begin{remark}\label{Remark:wpBSVIE}
\begin{enumerate}[label=$(\roman*)$, ref=.$(\roman*)$,wide, labelwidth=!, labelindent=0pt]
\item \label{Remark:wpBSVIE:ii} There are two noticeable differences between {\rm \Cref{Thm:wpbsvie}} and the results in the literature on {\rm type--I BSVIEs} \eqref{Eq:typeIBSVIE}, \eqref{Eq:typeIBSVIEextended0} and \eqref{Eq:typeIBSVIEextended}, as previously studied in {\rm \cite{yong2008well}}, {\rm \cite{wang2019time, hernandez2020me}} and {\rm \cite{hamaguchi2020extended,wang2020extended}}, respectively. The first is the additional terms, involving the derivative with respect to the parameter $s$ of the data, appearing in the \emph{a priori} estimates and the stability. The second one, is the space where the solution lives. Both differences are due to the fact that we are handling the diagonal term for $Z$ in the generator.\medskip

In fact, in light of {\rm \Cref{Prop:aprioriest:simplify}} and {\rm \Cref{Remark:contractionsimplify}} for the case of {\rm type--I BSVIEs} \eqref{Eq:typeIBSVIEextended}, i.e. where only the diagonal of $Y$ is allowed in the generator, one can work in the more standard $($compared to the existing literature$)$ space $(\Hf^{\star}, \|\cdot\|_{\Hf^{\star}})$ given by
\begin{gather*}
 \Hf^{\star}:=  \S^{2,2}(\R^{d}) \timesr \H^{2,2}(\R^{n\times d }) \timesr {\M^{2,2}}(\R^{d}), \; \|(Y,Z,N)  \|^2_{\Hf^{\star}}:=   \|Y\|_{\S^{2,2}}^2 + \|Z\|_{\H^{2,2}}^2+\|N \|_{\M^{2,2}}^2.
\end{gather*}
Then, the \emph{a priori} estimate \eqref{Eq:aprioriestBSVIE} simplifies to
\begin{align*}
\|( Y,Z,N) \|_{\Hf^\star}\leq C \big(   \|\xi \|^2_{\Lc^{2,2}} +\| \tilde f \|_{\L^{1,2,2}}^2 \big),
\end{align*}
and for $(Y^i,Z^i,N^i)$ the solution to {\rm type--I BSVIE} \eqref{Eq:typeIBSVIEextended} with data $(\xi^i,f^i)$ for $i\in \{1,2\}$, we obtain
\begin{align*}
\| (\delta Y,\delta Z,\delta N)  \|^2_{\Hf^\star} \leq &\ C\Big(  \|\delta \xi\big\|^2_{\Lc^2} +  \|\delta_1 f \|_{\L^{1,2}}^2 \Big).
\end{align*}

\item Lastly, we remark that our approach shows that $(Y,Z,N)$ satisfy the so--called {\rm M--property}, see {\rm\cite{yong2008well}}, that is
\[ Y_t^t=\E\big[Y_t^t\big]+\int_0^t {Z_r^t}^\t \d X_r.\]
Indeed, we have $Y_t^t=\E \big[ \xi(t) +\int_t^T f_r(t,X,Y_r^t, Z_r^t, \Yc_r,\Zc_r)   \d r\big| \Fc_t\big]$, so that in combination with the flow property of the first {\rm BSDE} in \eqref{Eq:systemBSDEf} we obtain
\[ \E\bigg[ \xi(t) +\int_0^T f_r(t,X,Y_r^r, Z_r^r, \Yc_r,\Zc_r)   \d r\bigg] = Y_t^t +\int_0^t   f_r(t,X,Y_r^r, Z_r^r, \Yc_r,\Zc_r) -\int_0^t {Z_r^t}^\t \d X_r.\]

\end{enumerate}
\end{remark}

\section{Type--I BSVIEs, parabolic PDEs and time--inconsistent control}\label{Section:PDEsandtimeinconsistency}

This section is devoted to the application of our results in \Cref{Section:BSVIE} to the problem of time--inconsistent control for sophisticated agents. Moreover, we also reconcile seemingly different approaches to the study of this problem.

\subsection{Representation formula for adapted solutions of type--I BSVIEs}

Building upon the fact that second--order, parabolic, semilinear PDEs of HJB type admit a non–linear Feynman--Kac representation formula, we can identify the family of PDEs associated to {\rm Type-I BSVIEs}. This is similar to the representation of forward backward stochastic differential equations, FBSDEs for short, see \cite{wang2019backward}.\medskip

For $(s,t,x,u,y,v,z,\upgamma,\Sigma)\in [0,T]^2\times \Xc\times (\R^d)^2\times (\R^{n\times d})^2\times (\R^{n\times n})^d\times \R^{n\times m}$, we define $\Tr\big[  \Sigma \Sigma^\t \gamma \big]\in \R^n$ by $\big(\Tr\big[  \Sigma \Sigma^\t \gamma \big]\big)_i := \Tr\big[  \Sigma \Sigma^\t \gamma_i \big],\; i\in \{1,\dots,d\}.$ Let $f$ and $\sigma$ be as in the preceding section, and $b:[0,T]\times \Xc \longrightarrow \R^m$ be bounded, $b_\cdot(X)\in \Pc_{\rm meas}(\R^m,\F) $ and 
\begin{align*}
G \big(s,t,x,u,v,y,z,\gamma \big) :=  v^\t\sigma(t,x)  b (t,x)+\frac{1}{2}\Tr[  \sigma\sigma^\t (t,x) \gamma  ]+ f(s,t,x,u,v,y,z).
\end{align*}

\begin{theorem}\label{Thm:rep}
Consider the Markovian setting, i.e. $\varphi_t(X,\cdot)=\varphi_t(X_t,\cdot)$ for $\varphi\in\{b,\sigma,f,\partial_s f\}$, and $\varphi(s,X)=\varphi(s,X_T)$ for $\varphi\in\{\xi,\partial_s \xi\}$. Assume that 

\begin{enumerate}[label=$(\roman*)$, ref=.$(\roman*)$,wide, labelwidth=!, labelindent=0pt]
\item for $(s,t,x)\in [0,T)\times[0,T]\times \R^n=:\Oc$, there exists $\vc\in  \Cc_{1,1,2}^d([0,T]^2\times \R^n)$ classical solution to
\begin{align*}\begin{cases}
\partial_t \vc(s,t,x) +G\big(s,t,x,\vc(s,t,x),\partial_{x} \vc(s,t,x),\vc(t,t,x),\partial_{x} \vc(t,t,x),\partial_{x x}^2\vc(s,t,x)  \big)=0,\; (s,t,x)\in \Oc ,\\[0.7em]
 \vc(s,T,x)=\xi(s,x), \; (s,x)\in [0,T]\times \R^n.
\end{cases}
\end{align*}
\item $\vc(s,t,x)$ and $\partial_x \vc(s,t,x)$ have uniform exponential growth in $x$\footnote{Here, $|\cdot|_1$ denotes the $\ell_1$ norm in $\R^n$, i.e. for $x\in \R^n$, $|x|_1:=\sum_{i=1}^d|x_i|$}, i.e. $\exists C>0$ such that for all $(s,t,x)\in [0,T]^2\times \Xc$
\[ |\vc(t,x)|+|\partial_x \vc(t,x)|\leq C\exp(C |x|_1).\]
\end{enumerate}

Then, $ Y_t^s:=\vc(s,t,X_t)$, and $Z_t^s:=\partial_x \vc(s,t,X_t)$ define a solution to the {\rm type--I BSVIE} given by
\begin{align}\label{Eq:BSDErep}
Y_t^s= Y^s_T+ \int_t^T \big( f _r(s,X_r,Y_r^s,Z_r^s,Y_r^r,Z_r^r)+{Z_r^s}^\t \sigma_r(X_r)b_r(X_r)\big) \d r -\int_t^T {Z_r^s}^\t \d X_r.
\end{align}

\end{theorem}
\begin{proof}
Let $s\in [0,T]$ and $\P$ as in \Cref{Sec:Pstatement}. Applying It{\= o}'s formula to the process $Y_t^s$ we find that $\P\as$
\begin{align*}
Y_t^s& =Y^s_T-\int_t^T\bigg( \partial_t \vc(s,r,X_r)+\frac{1}{2}\Tr[\sigma \sigma^\t_r(X_r) \partial^2_{x x} \vc(s,r,X_r)]\bigg) \d r -\int_t^T {Z_r^s}^\t \d X_r\\
& = Y^s_T+ \int_t^T \big( f _r(s,X_r,Y_r^s,Z_r^s,Y_r^r,Z_r^r)+{Z_r^s}^\t \sigma_r(X_r)b_r(X_r)\big) \d r -\int_t^T {Z_r^s}^\t \d X_r.
\end{align*}

We are left to check the integrability of $Y$ and $Z$. As $\sigma$ is bounded, it follows that $X_t$ has exponential moments of any order which are bounded on $[0,T]$, i.e. $\exists C>0$, such that $ \sup_{t\in[0,T]} \E^\P[\exp ( c |X_t|_1)] \leq  C < \infty$, for any $c > 0$, where $C$ depends on $T$ and the bound on $\sigma$. This together with the growth condition on $\vc(s,t,x)$ and $\partial_x \vc(s,t,x)$ yield the integrability.
\end{proof}

\begin{remark}
The previous result provides a representation for a particular class of {\rm type--I BSVIEs} as in \eqref{Eq:BSDErep}, i.e. with an additional term linear in $z$. This is a consequence of the dynamics of $X$ under $\P$, see {\rm \Cref{Section:stochasticbasis}}. Nevertheless, as $b$ is bounded, we can define $\P^b\in \Prob(\Xc)$, equivalent to $\P$, given by
\begin{align*}
\frac{\d \P^b}{\d \P}:=\exp \bigg( \int_0^T b_r(X_r)\cdot \d B_r-\int_0^T |b_r(X_r) |^2 \d r\bigg).
\end{align*} 
Girsanov's theorem then shows that $B^b:=B-\int_0^\cdot b_r(X_r)\d r$ is a $\P^b$--Brownian motion and
\[ X_t=x_0+\int_0^t \sigma_r(X_r) b_r(X_r)\d r + \int_0^t \sigma_r(X_r)\d B_r^b,\; t\in [0,T],\; \P^b\as,\]
and consequently
\begin{align*}
Y_t^s=Y^s_T+\int_t^T  f _r(s,X_r,Y_r^s,Z_r^s,Y_r^r,Z_r^r)  \d r -\int_t^T Z_r^\t \d X_r,\; t\in [0,T],\; \P\as
\end{align*}
\end{remark}

\subsection{On the characterisation of equilibria and their value function for time--inconsistent control problems}
The systematic study of time--inconsistent stochastic control problems in continuous--time for sophisticated agents started with the Markovian setting, and is grounded in the notion of equilibrium first proposed in \citet*{ekeland2008investment}, \citet*{ekeland2010golden} and the further analysis of \citet*{bjork2017time}. The main contribution of \cite{bjork2017time} was to provide an infinite dimensional system of PDEs, or Hamilton--Jacobi--Bellman equation, that identifies the value function and an equilibrium policy, see \Cref{Eq:HJB-BKM} below. Soon after, \citet*{wei2017time} presented a verification argument via a one dimensional PDE, but over an extended domain, see \Cref{Eq:HJB-WY} below. Both approaches have generated independent lines of research in the community, including both analytic and probabilistic methods, but no compelling connections have been established, as far as we know. \medskip

As discussed in the introduction, BSDEs and BSVIEs appear naturally as part of the probabilistic study of time--inconsistent stochastic control problems. Some of the advantages of this approach are the possibility to extend the study to the non--Markovian framework, and to allow for rewards functionals inspired by the concept of recursive utilities. Indeed, the approaches in \cite{hernandez2020me} and \cite{wang2019time} address these directions, and can be regarded as extensions of \cite{bjork2017time} and \cite{wei2017time}, respectively. As such, it is not surprising that in order to characterise an equilibrium and its associated value function, both \cite{hernandez2020me} and \cite{wang2019time} lay down an infinite dimensional systems of BSDEs, and a type--I BSVIEs, respectively.  In fact, \cite[Theorem 3.7]{hernandez2020me} and \cite[Theorem 5.1]{wang2019time} establish representation formulae for the analytic, i.e. PDEs, counterparts. \medskip

Nevertheless, as part of the analysis of time--inconsistent control through BSDEs, \cite{hernandez2020me} noticed that their approach led to the well--posedness of a BSVIE. This is nothing but a manifestation of our results in \Cref{Section:BSVIE}. Indeed, one of the virtues of \Cref{Thm:wpbsvie} is that it reconciles, at the probabilistic level, the findings of \cite{hernandez2020me} and \cite{wang2019time}. We will detail this relationship in the following. Moreover, we also include \Cref{Thm:eqvpds} which does the same at the PDE level. To sum up, we can visualise this in the next picture. 

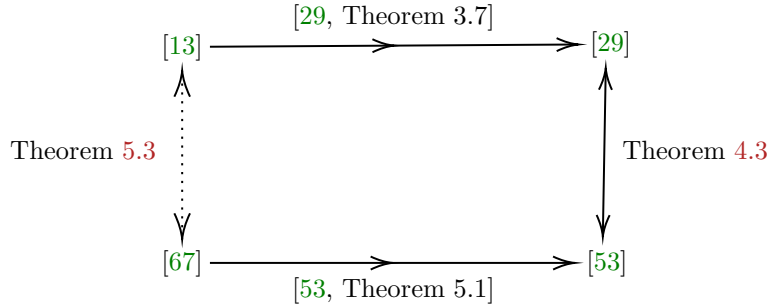
\begin{center}
\tikzset{every picture/.style={line width=0.75pt}} 
\begin{tikzpicture}[x=0.75pt,y=0.75pt,yscale=-1,xscale=1]
\draw (113.62,63.35) node   {\cite{bjork2017time}};
\draw (220.18,48.25) node    {\cite[Theorem 3.7]{hernandez2020me}};
\draw (327.42,62.04) node     {\ \cite{hernandez2020me}};
\draw (63.62,115.25) node    {{\rm \Cref{Thm:eqvpds}}};
\draw (372.3,115.25) node    {{\rm \Cref{Thm:wpbsvie}}};
\draw (113.62,172.46) node    {\cite{wei2017time}};
\draw (220.18,186.25) node    {\cite[Theorem 5.1] {wang2019time}};
\draw (325.62,172.46) node    {\ \cite{wang2019time}};

\draw [line width=0.75]    (127.62,63.26) -- (311.42,62.14) ;
\draw [shift={(313.42,62.13)}, rotate = 539.65] [color={rgb, 255:red, 0; green, 0; blue, 0 }  ][line width=0.75]    (10.93,-3.29) .. controls (6.95,-1.4) and (3.31,-0.3) .. (0,0) .. controls (3.31,0.3) and (6.95,1.4) .. (10.93,3.29)   ;
\draw [shift={(220.52,62.69)}, rotate = 539.65] [color={rgb, 255:red, 0; green, 0; blue, 0 }  ][line width=0.75]    (10.93,-3.29) .. controls (6.95,-1.4) and (3.31,-0.3) .. (0,0) .. controls (3.31,0.3) and (6.95,1.4) .. (10.93,3.29)   ;
\draw [line width=0.75]    (127.62,172.37) -- (309.62,172.37) ;
\draw [shift={(311.62,172.37)}, rotate = 539.65] [color={rgb, 255:red, 0; green, 0; blue, 0 }  ][line width=0.75]    (10.93,-3.29) .. controls (6.95,-1.4) and (3.31,-0.3) .. (0,0) .. controls (3.31,0.3) and (6.95,1.4) .. (10.93,3.29)   ;
\draw [shift={(219.62,172.37)}, rotate = 539.65] [color={rgb, 255:red, 0; green, 0; blue, 0 }  ][line width=0.75]    (10.93,-3.29) .. controls (6.95,-1.4) and (3.31,-0.3) .. (0,0) .. controls (3.31,0.3) and (6.95,1.4) .. (10.93,3.29)   ;
\draw [line width=0.75]    (327.22,76.04) -- (325.81,157.15) ;
\draw [shift={(325.78,159.15)}, rotate = 270.8] [color={rgb, 255:red, 0; green, 0; blue, 0 }  ][line width=0.75]    (10.93,-3.29) .. controls (6.95,-1.4) and (3.31,-0.3) .. (0,0) .. controls (3.31,0.3) and (6.95,1.4) .. (10.93,3.29)   ;
\draw [shift={(327.25,74.04)}, rotate = 90.8] [color={rgb, 255:red, 0; green, 0; blue, 0 }  ][line width=0.75]    (10.93,-3.29) .. controls (6.95,-1.4) and (3.31,-0.3) .. (0,0) .. controls (3.31,0.3) and (6.95,1.4) .. (10.93,3.29)   ;
\draw [line width=0.75]  [dash pattern={on 0.84pt off 2.51pt}]  (113.62,77.35) -- (113.62,158.46) ;
\draw [shift={(113.62,160.46)}, rotate = 270] [color={rgb, 255:red, 0; green, 0; blue, 0 }  ][line width=0.75]    (13.12,-3.95) .. controls (8.34,-1.68) and (3.97,-0.36) .. (0,0) .. controls (3.97,0.36) and (8.34,1.68) .. (13.12,3.95)   ;
\draw [shift={(113.62,75.35)}, rotate = 90] [color={rgb, 255:red, 0; green, 0; blue, 0 }  ][line width=0.75]    (13.12,-3.95) .. controls (8.34,-1.68) and (3.97,-0.36) .. (0,0) .. controls (3.97,0.36) and (8.34,1.68) .. (13.12,3.95)   ;
\end{tikzpicture}
\end{center}

Let $A\subseteq \R^p$ be a compact set, we introduce the mappings 
\begin{align*}
 & \bar f: [0,T]^2\times \R^n \times A  \longrightarrow \R ,\; f_\cdot(s,\cdot ,a )\in \Pc_{{\rm prog}}(\R,\F),\text{ for every } (s,a)\in [0,T]\times A,\\
 & b: [0,T]\times \R^n \times A \longrightarrow \R^m, \text{ bounded, } \; b_\cdot(\cdot,a) \in \Pc_{{\rm meas}}(\R^m,\F) \text{ for every } a\in A.
\end{align*}
With this we may define
\begin{align*}
& \bar g (s,t,x,a,v):= \bar f(s,t,x,a)+v\cdot  \sigma(t,x) b(t,x,a),\; H(s,t,x,v):= \sup_{a\in A} \, \bar g (s,t,x,a,v),\\
&\nabla \bar g (s,t,x,a,{\rm v}):=  \partial_s \bar f(s,t,x,a)+{\rm v} \cdot  \sigma(t,x) b(t,x,a),
\end{align*}
and assume there exists $a^\star:[0,T]^2\times \R^n \times \R^n\longrightarrow A$, measurable, such that, $\bar g (s,t,x,a^\star(s,t,x,v),v)=H(s,t,x,v)$.\medskip

Following the approach of \cite{wang2019time}, let us assume that given an admissible $A$--valued strategy $\nu$ over the interval $[s,T]$, the reward at $s\in [0,T]$ is given by the value at $s$ of the $Y$ coordinate of the solution to the type--I BSVIE given by

\[ Y_t^\nu=\xi(t,X_T)+\int_t^T \tilde g_r(t,X_r,\nu_r,Z_r^t) \d r-\int_t^T Z_r^t\cdot  \d X_r,\; t\in [s,T],\; \P\as \]

This problem is time--inconsistent, and we are interested in finding an \emph{equilibrium policy}. \cite{wang2019time} finds that the value of the game coincides with the $Y$ coordinate of the following type--I BSVIE
\[ Y_t=\xi(t,X_T)+\int_t^T \bar g_r\big(t,X_r,a^\star(r,r,X_r,Z_r^r),Z_r^t\big)\d r-\int_t^T Z_r^t \cdot \d X_r,\; t\in [0,T],\; \P\as,  \]
where the diagonal of $Z$ appears in the generator. However, decoupling the dependence between the time variable and the variable source of time--inconsistency, we can define

\[ Y_t^s:=\xi(s,X_T)+\int_t^T \bar g_r\big(s,X_r,a^\star(r,r,X_r,Z_r^r),Z_r^s\big)\d r-\int_t^T Z_r^s\cdot \d X_r, \;  0\leq s \leq t \leq T,\; \P\as  \]

\Cref{Thm:wpbsvie} shows the equivalence of this approach to that of \cite{hernandez2020me}, based on the system
\begin{align*}
Y_t& =\xi(T,X_T)+\int_t^T \big( H_r(r,X_r,Z_r)-\partial Y_r^r\big) \d r-\int_t^T Z_r \cdot \d X_r,\; t\in [0,T],\; \P\as\\
\partial Y_t^s& =\partial_s \xi(s,X_T)+\int_t^T \nabla \bar g_r(s,X_r,a^\star(r,r,X_r,Z_r^r),\partial Z_r^s)\d r-\int_t^T \partial Z_r^s \cdot \d X_r,\;  0\leq s \leq t \leq T,\; \P\as
\end{align*}

We now move on to establish the connection of the analyses at the PDE level. The original result of \cite{bjork2017time} is based on the semi--linear PDE system of HJB type given for $\big(V(t,x),\Jc(s,t,x) \big) \in  \Cc_{1,2}([0,T] \times \R^n)\times \Cc_{1,1,2}([0,T]^2\times \R^n)$ by
\begin{align}\label{Eq:HJB-BKM}
\begin{cases}
  \displaystyle \partial_t V(t,x)+\frac{1}2 \Tr[\sigma\sigma^\t(t,x) \partial_{xx}V(t,x)]+  H(t,t,x,\partial_{x}V(t,x))- \partial_s\Jc(t,t,x)=0,\; (s,t,x)\in\Oc,   \\[0.7em]
  \displaystyle   \partial_t \Jc(s,t,x)+ \frac{1}2 \Tr[\sigma\sigma^\t(t,x) \partial_{xx} \Jc(s,t,x)]+  \bar g\big (s,t,x,  \partial_x \Jc(s,t,x),a^\star\big(t,t,x,\partial_x V(t,x)\big)\big)=0,\; (s,t,x)\in\Oc,\\[0.7em]
  \displaystyle  V(T,x)=\xi(T,x),\;  \Jc(s,T,x)= \xi(s,x),\; (s,x)\in [0,T]\times \R^d.
\end{cases}
\end{align}
On the other hand, \cite{wei2017time} considers the equilibrium HJB equation for $ \Jc(s,t,x) \in  \Cc_{1,1,2}([0,T]^2\times \R^n)$ given by
\begin{align}\label{Eq:HJB-WY}
\begin{cases}
  \displaystyle   \partial_t \Vc(s,t,x)+ \frac{1}2 \Tr[\sigma\sigma^\t(t,x) \partial_{xx} \Vc(s,t,x)]+  \bar g\big(s,t,x,  \partial_x \Vc(t,t,x), a^\star\big(t,t,x,\partial_x \Vc(t,t,x)\big)\big)=0,\; (s,t,x)\in\Oc,\\[0.7em]
  \displaystyle  \Vc(s,T,x)= \xi(s,x),\; (s,x)\in [0,T]\times \R^d.
\end{cases}
\end{align}

By setting $\Vc(s,t,x)=\Jc(s,t,x)$, it is immediate that a solution to \eqref{Eq:HJB-BKM} defines a solution to \eqref{Eq:HJB-WY}. The next theorem establishes the converse result.

\begin{theorem}\label{Thm:eqvpds}
Suppose \eqref{Eq:HJB-BKM} and \eqref{Eq:HJB-WY} are well--posed.
\begin{enumerate}[label=$(\roman*)$, ref=.$(\roman*)$,wide, labelwidth=!, labelindent=0pt]
\item Let $\Jc(s,t,x) \in  \Cc_{1,1,2}([0,T]^2\times \R^n)$ solve \eqref{Eq:HJB-WY}. Then $\Vc(s,t,x):=\Jc(s,t,x)$ solves \eqref{Eq:HJB-WY}.
\item Let $\Vc(s,t,x) \in  \Cc_{1,1,2}([0,T]^2\times \R^n)$ solve \eqref{Eq:HJB-WY}. Then $\big(V(t,x),\Jc(s,t,x) \big):=\big(\Vc(t,t,x),\Vc(s,t,x)\big)$ solves \eqref{Eq:HJB-BKM}.
\end{enumerate}
\end{theorem}
\begin{proof}
We are only left to argue $(ii)$. It is clear $\big(V(t,x),\Jc(s,t,x) \big) \in  \Cc_{1,2}([0,T] \times \R^n)\times \Cc_{1,1,2}([0,T]^2\times \R^n)$, the results follows as $- \partial_t V(t,x)+\partial_s \Jc (t,t,x) = - \partial_t \Vc(t,t,x).$
\end{proof}

\section{Analysis of the BSDE system}\label{Section:Analysis}
Let us first recall the elementary inequalities
\begin{align}\label{Eq:ineqsquare}
\bigg(\sum_{i=1}^n a_i\bigg)^2\leq n\sum_{i=1}^n a_i^2,
\end{align}
valid for any positive integer $n$ and any collection $(a_i)_{1\leq i\leq n}$ of non--negative numbers, as well as, Young's inequality which guarantees that for any $\eps >0$, $2ab\leq \eps a^2 +\eps^{-1} b^2$.\medskip

In order to alleviate notations, and as it is standard in the literature, we suppress the dependence on $\omega$, i.e. on $X$, in the functionals, and, write $\E$ instead of $\E^\P$ as the underlying probability measure $\P$ is fixed. Moreover, we will write $\I^2$ instead of $\I^2 (E)$ for any of the integrability spaces involved, the specific space $E$ is fixed and understood without ambiguity.

\subsection{Regularity of  the system and the diagonal processes}

In preparation to the proof of \Cref{Thm:wp}, we present next a couple of lemmata from which we will benefit in the following. As a historical remark, we mention the following is in the spirit of the analysis in {\rm \citet*[Section 3]{protter1985volterra}} and {\rm \citet*{pardoux1990stochastic}} of forward Volterra integral equations.\medskip

A technical detail in our analysis is to identify appropriate spaces so that given $(\partial U,U,V,M   )$ one can rigorously define the processes $\big((U_t^t)_{t\in[0,T]},(V_t^t)_{t\in[0,T]},(M_t^t)_{t\in[0,T]}, (\partial U_t^t)_{t\in[0,T]}\big)$. It is known that for $ U\in \S^{2,2}$, the diagonal process $( U_t^t)_{t\in [0,T]}$ is well--defined. Indeed, this follows from the pathwise regularity of $U^s$ for $s\in [0,T]$ and has been noticed since {\rm \cite{hamaguchi2020extended,hernandez2020me, wang2017optimal2}}. The same argument works for $(\partial U,M)\in \S^{2,2}  \times {\M^{2,2}}$. Unfortunately, the same reasoning cannot be applied for arbitrary $V\in \H^{2,2}$ and motives the introduction of the space $\Ho$, see \Cref{Remark:wpsystem}.

\begin{lemma}\label{Lemma:partialU}
Let {\rm \Cref{AssumptionA}} hold and $(\Yc,\Zc)\in \L^2\times \H^2$. Let $(U,V,M)\in\L^{2,2} \times \H^{2,2} \times \M^{2,2}$ be a solution to
\begin{align*}
 U_t^s&=   \eta (s)+\int_t^T  g_r(s,U_r^s,V_r^s, \Yc_r, \Zc_r) \d r-\int_t^T   {V_r^s}^\t  \d X_r-\int_t^T \d   M^s_r, \; (s,t)\in [0,T]^2,\;  \P\as
\end{align*}
\begin{enumerate}[label=$(\roman*)$, ref=.$(\roman*)$,wide, labelwidth=!, labelindent=0pt]
\item \label{Lemma:partialU:i} There exists $(\partial U , \partial V, \partial M)\in  \S^{2,2} \times  \H^{2,2} \times \M^{2,2}$ unique solution to 
\begin{align}\label{Eq:BSDEderivative}
\partial U_t^s&=  \partial_s \eta (s)+\int_t^T  \nabla g_r(s,\partial U_r^s,\partial V_r^s,U_r^s,V_r^s, \Yc_r, \Zc_r) \d r-\int_t^T\partial  {V_r^s}^\t  \d X_r-\int_t^T \d \partial M^s_r, \; (s,t)\in [0,T]^2,\;  \P\as ;
\end{align}
\item \label{Lemma:partialU:ii} there exist $C>0$, such that for all $c>4L_{ g}$ and $(s,t)\in [0,T]^2$
\begin{align}\label{Eq:aprioriest:partialUV}
\begin{split}
 \E\bigg[ \int_t^T\e^{cr}  |\sigma_r^\t \partial V_r^s|^2 \d r\bigg] & \leq  C \E\bigg[\e^{cT}  \big| \partial_{ s} \eta(s)\big|^2+ \bigg( \int_t^T \e^{\frac{c}2 r} |\nabla \tilde g_r(s)|\d r\bigg)^2 \\
&\hspace{3em}+\int_t^T\e^{c r} \big(|U_r^s|^2 +|\sigma_r^\t V_r^s|^2 +|\Yc_r|^2+|\sigma_r^\t  \Zc_r|^2\big) \d r \bigg];
\end{split}
\end{align}

\item \label{Lemma:partialU:iii} \eqref{Eq:constraintspace} holds, i.e. for any $s\in [0,T]$ 
\begin{align*}\tag{$\Dc$}
\bigg(\int_s^T \partial U^r \d r,\int_s^T \partial V^r \d r,\int_s^T \partial M^r \d r\bigg)=\big(U^T-U^s,V^T-V^s,M^T-M^s \big), \text{ in } \S^2\times \H^2\times \M^2;
\end{align*}

\item \label{Lemma:partialU:iv} $V\in \Ho$. Moreover, for $\Vc:=(V_t^t)_{t\in [0,T]}$ and $\eps>0$, $\P\as$
\begin{align}\label{Eq:ineqVtt2:0}
 \int_t^T |\sigma^\t_u \Vc_u|^2 \d u \leq\int_t^T | \sigma^\t_uV_u^t|^2 \d u+ \int_t^T\int_r^T \eps |\sigma_u^\t  V^r_u|^2+ \eps^{-1} |\sigma_u^\t \partial V^r_u |^2 \d u \d r,\ t\in [0,T].
\end{align}
\end{enumerate}
\end{lemma}
\begin{proof}
Note that in light of \Cref{AssumptionA}\ref{AssumptionA:i}, for $(t,s,x ,u, v ,y,z)\in [0,T]^2\times \Xc \times \R^{d_2}\times \R^{d_2\times n}\times \R^{d_1}\times \R^{d_1\times n}$, $({\rm  u} ,{\rm v})\longmapsto \nabla g_t(s,x, {\rm u}, {\rm v} ,u, v ,y,z)$ is linear. Therefore, for any $s\in [0,T]$ the second equation defines a linear BSDE, in $(\partial U^s, \partial V^s)$, whose generator at zero, by \Cref{AssumptionA}\ref{AssumptionA:iii}, is in $\L^{1,2}$. Therefore, its solution $(\partial U^s, \partial V^s, \partial M^s)\in \S^2\times \H^2\times \M^2$ is well--defined from classic results, see for instance \citet*{zhang2017backward} or \cite{el1997backward}. The continuity of the applications $s\longmapsto (\partial U^s,\partial V^s, \partial M^s)$, e.g. $([0,T],\Bc([0,T])) \longrightarrow (\S^{2},\|\cdot \|_{ \S^{2}}): s \longmapsto \partial U^{s}$, follows from the classical stability results of BSDE, and that by assumption $(U, V,Y,Z)\in \L^{2,2}\times \H^{2,2}\times \L^{2}\times \H^{2}$ and $s\longmapsto (\partial_s \eta(s), \nabla \tilde g(s))$ is continuous. This establishes $(\partial U, \partial V, \partial M)\in \S^{2,2}\times \H^{2,2}\times \M^{2,2}$.\medskip

$(ii)$ follows from classic a priori estimates, but when the norms are considered over $[t,T]$ instead of $[0,T]$. Indeed, following the argument in \cite[Proposition 2.1]{el1997backward}, applying It{\=o}'s formula to $\e^{ct}|\partial U_t^s|^2$ we may find $C>0$ such that for any $c> 4 L_{ g} $ and $(s,t)\in [0,T]^2$
\begin{align*}
& \E\bigg[\sup_{r\in [t,T]} \Big\{\e^{cr} \big| \partial U_r^s\big|^2\Big\} +\int_t^T \e^{cr} \big|\sigma^\t_r  \partial V_r^s\big|^2 \d r+\int_t^T \e^{cr-} \d [\partial M^s]_r \bigg]\\
&\;  \leq  C \E\bigg[\e^{cT} \big| \partial_{s} \eta(s)\big|^2+\bigg ( \int_t^T \e^{\frac{c}2 r} \big| \nabla g_t(s, 0, 0,U_r^s, V_r^s ,\Yc_r,\Zc_r) \big| \d r\bigg)^2 \bigg]\\ 
&  =  C \E\bigg[\e^{cT}  \big| \partial_{s} \eta(s)\big|^2+\bigg ( \int_t^T  \e^{\frac{c}2 r} \Big| \partial_s g_r(s,U_r^s,V_r^s,\Yc_r,\Zc_r)\Big| \d r \bigg)^2\bigg]\\
&\; \leq  C \E\bigg[\e^{cT}  \big| \partial_{ s} \eta(s)\big|^2+ \bigg( \int_t^T \e^{\frac{c}2 r} |\nabla \tilde g_r(s)|\d r\bigg)^2  +\int_t^T\e^{c r} \big(|U_t^s|^2 +|\sigma_t^\t V_t^s|^2 +|\Yc_t|^2+|\sigma_t^\t  \Zc_t|^2\big) \d t \bigg],
\end{align*}
where in the second inequality we exploited the fact $(u,v,y,z)\longmapsto \partial_s g (t,s,x,u,v,y,z)$ is Lipschitz, see \Cref{AssumptionA}\ref{AssumptionA:iii}, and $C>0$ was appropriately updated. \medskip

Next, we assume $(iii)$ and show $(iv)$. We know $V\in \H^{2,2}$ we are left to verify $\Vc\in \H^2$. The fact $\Vc$ is well--defined follows from $(iii)$. Indeed, for any $s\in [0,T]$
\[ V^s_t = V^T_t-\int_s^T \partial V^r_t \d r , \; \d t \otimes \d \P\ae \text{ \rm in } [0,T]\times \Xc, \]
which together with the continuity, for any $(t,x)\in [0,T]\times \Xc$, of the mapping $([0,T],\Bc([0,T])) \longrightarrow (\R^{n\timesr d_2}, |\cdot | ): s \longmapsto \int_s^T \partial V^r_t(x) \d r $, shows that
we can define $ V_t^t,\;  \d t\otimes \d \P\ae \text{ \rm in } [0,T]\times \Xc$.\medskip

We now verify \eqref{Eq:ineqVtt2:0}. Indeed, Fubini's theorem and Young's inequality yields that for $\eps>0$ 
\begin{align*}
\int_t^T | \sigma^\t_uV_u^u|^2-|\sigma^\t_u V_u^t|^2 \d u =\int_t^T  \int_t^u  2 \Tr\Big [ {V_u^r}^\t{\sigma_u}   {\sigma^\t_u} \partial V_u^r\Big] \d r \d u & = \int_t^T \int_r^T 2 \Tr\Big [ {V_u^r}^\t{\sigma_u}   {\sigma^\t_u} \partial V_u^r   \Big ] \d u \d r\\
& \leq \int_t^T\int_r^T \eps |\sigma_u^\t  V^r_u|^2+ \eps^{-1} |\sigma_u^\t \partial V^r_u |^2 \d u \d r .
\end{align*} 

Thus $\| \Vc\|_{\H^2}<\infty$ and consequently $V\in \Ho$. This proves $(iv)$. \medskip

We now argue $(iii)$. We know the mapping $[0,T]\ni s\longmapsto (\partial Y^s,\partial Z^s,\partial M^s)$ is continuous, in particular integrable. A formal integration with respect to $s$ to \eqref{Eq:BSDEderivative} leads to
\begin{align*}
\int_s^T\partial U_t^r \d r & = \int_s^T \partial_s \xi(r)\d r + \int_s^T \int_t^T   \partial_s g_u 	(r,U_u^r,V_u^r,\Yc_u,\Zc_u)  + \partial_y g_u 	(r,U_u^r,V_u^r,\Yc_u,\Zc_u) \partial U_u^r \d u \d r \\
&\quad +\int_t^T   \int_s^T \sum_{i=1}^n \partial_{v_i} g_u 	(r,U_u^r,V_u^r,\Yc_u,\Zc_u) (\partial V_u^r)_{: i}  \d u\d r -\int_s^T \int_t^T   {\partial Z_u^r}^\t \d X_u\d r -\int_t^T \int_s^T \d \partial M_u^r\d r.
\end{align*}

Let $(\Pi^\ell)_{\ell}$ be a properly chosen sequence of partitions of $[s,T]$, as in \citet*[Theorem 1]{neerven2002approximating}, $\Pi^\ell=(s_i)_{i=1,\dots,n_\ell}$ with $\| \Pi^\ell\|:=\sup_{i}|s_{i+1}-s_i| \leq \ell$. To ease the notation, we set $\Delta s_i^\ell:=s_i^\ell-s_{i-1}^\ell$, and, for a generic family process $(\phi^s)_{s\in [0,T]}$, and a mapping $s\longmapsto \partial_s \xi(s,x)$, we define 
\begin{align*}
I^\ell( \phi ):=\sum_{i=0}^{n_\ell} \Delta s_i^\ell  \phi^{s_i^\ell}, \; \delta \phi:=\phi^T-\phi^s,\; I^{\ell}(\partial_s \xi(\cdot,x)):=\sum_{i=0}^n \Delta s_i^\ell \partial_s \xi(s_i^\ell,x),  \; \delta \xi(x) := \xi(T,x)-\xi(s,x).
\end{align*}
We then notice that for any $t\in[0,T]$
\begin{align*}
I^\ell(\partial U)_t-(\delta U)_t  & =   I^\ell(\partial_s \xi(\cdot ))-\delta \xi + \int_t^T \big[I^\ell\big(\partial_s g_u(\cdot,U^\cdot_u, V^\cdot_u, \Yc_u,\Zc_u)\big )+I^\ell\big(\partial_y g_u(\cdot,U^\cdot_u, V^\cdot_u, \Yc_u,\Zc_u)\partial U^\cdot_u \big ) \big]\d u \\
&\quad  +\int_t^T \big[I^\ell\big(\partial_z g_u(\cdot,U^\cdot_u, V^\cdot_u, \Yc_u,\Zc_u)\partial V^\cdot_u\big ) -\delta g_u(\cdot,U^\cdot_u, V^\cdot_u, \Yc_u,\Zc_u) \big]\d u\\
&\quad -\int_t^T\big(I^\ell(\partial V)_r-(\delta V)_r\big)^\t\d X_u -I^\ell(\partial M_T^\cdot-\partial M_t^\cdot)-\delta (M_T-M_t).
\end{align*}
We now note that the integrability of $(\partial U, \partial V)$ and $(U, V)$ yields $I^\ell(\partial U)-(\delta U)\in \L^{2,2}$ and $I^\ell(\partial V)-(\delta V)\in \H^{2,2}$. Therefore, \citet*[Theorem 2.2]{bouchard2018unified} yields
\begin{align*}
&\|I^\ell(\partial U)-(\delta U) \|_{\L^{2,2}}^2 + \|I^\ell(\partial V)-(\delta V) \|_{\H^{2,2}}^2+ \|I^\ell(\partial M)-(\delta V) \|_{\M^{2,2}}^2 \\
&\; \leq C  \E \bigg[ \big| I^\ell(\partial_s \xi(\cdot ))-\delta \xi  \big|^2 +\bigg( \int_t^T \Big|I^\ell\big(\partial_s g_u(\cdot,U^\cdot_u, V^\cdot_u, \Yc_u,\Zc_u)\big ) -\delta g_u(\cdot,U^\cdot_u, V^\cdot_u, \Yc_u,\Zc_u) \Big| \d u \bigg)^2\bigg].
\end{align*}
To conclude, we first note that by our choice of $(\Pi^\ell)_\ell$, $I^\ell(\partial U)$ converges to the Lebesgue integral of $U^s$. In addition, the uniform continuity of $s\longmapsto \partial_s \xi(s,x)$ and $s\longmapsto\partial_s g(s,x,u,v,y,z)$, see \Cref{AssumptionA}\ref{AssumptionA:i}, justifies, via bounded convergence, the convergence in $\L^{2,2} $ (resp. $\H^{2,2}$) of $I^\ell(\partial U^s)$ to $U^T-U^s$ (resp. $I^\ell(\partial V^s)$ to $V^T- V^s$) as $\ell \longrightarrow 0$. The result follows in virtue of the uniqueness of $(U,V,M)$.
\end{proof}

The next lemma identifies the dynamics of $(U_t^t)_{t\in[0,T]}$.

\begin{lemma}\label{Lemma:intdiagpartialU}
Let $(\Yc,\Zc)\in \L^2\times \H^2$ and $(U,\partial U, V,\partial V,M, \partial M)\in (\L^{2,2})^2 \times \Ho \times \H^{2,2} \times (\M^{2,2})^2 $ be a solution to 
\begin{align*}
\begin{split}
U_t^s&=   \eta (s,X_{\cdot\wedge T})+\int_t^T  g_r(s,X,U_r^s,V_r^s, \Yc_r, \Zc_r) \d r-\int_t^T {V_r^s}^\t  \d X_r-\int_t^T \d M^s_r,\; (s,t)\in[0,T]^2,  \\
\partial U_t^s&=  \partial_s \eta (s,X_{\cdot\wedge T})+\int_t^T  \nabla g_r(s,X,\partial U_r^s,\partial V_r^s,U_r^s,V_r^s, \Yc_r, \Zc_r) \d r-\int_t^T\partial  {V_r^s}^\t  \d X_r-\int_t^T \d \partial M^s_r,\; (s,t)\in[0,T]^2.
\end{split}
\end{align*}
Then, 
\begin{align*}
U_t^t=U_T^T+\int_t^T \big( g_r(r,U_r^r,V_r^r,\Yc_r,\Zc_r)-\partial U_r^r\big) \d r -\int_t^T V_r^r  \d X_r-\int_t^T \d \widetilde M_r,\ t\in [0,T],\ \P\as
\end{align*}
where $\widetilde M :=(\widetilde M_t)_{t\in [0,T]}$ is given by $\widetilde M_t := M_t^t-\int_0^t \partial M_r^r \d r$. Moreover, $\widetilde M\in \M^2$.
\end{lemma}
\begin{proof}
We show that $\P\as$, for any $t\in [0,T]$
\begin{align*}
\int_t^T \partial U_r^r \d r= U_T^T-U_t^t+\int_t^T g_r(r,U_r^r,V_r^r,\Yc_r,\Zc_r)\d r -\int_t^T V_r^r \d X_r-\int_t^T \d \widetilde M_r.
\end{align*}
Indeed, note that, $\P\as$, for any $t\in [0,T]$
\begin{align*}
\int_t^T \partial U_r^r \d r=& \int_t^T \partial_s \eta (r) \d r + \int_t^T\bigg( \int_r^T \nabla g_u(r,\partial U_u^r, \partial V_u^r , U_u^r, V_u^r,\Yc_u,\Zc_u)\d u -\int_r^T \partial {V_u^r}^\t  \d X_u - \int_r^T \d \partial M_u^r \bigg) \d r.
\end{align*}

By Fubini's theorem, the change of variables formula for the Lebesgue integral, \cite[Theorem 54]{protter2005stochastic}, and \Cref{Lemma:partialU} we have that, $\P\as$, for any $t\in [0,T]$
\begin{align*}
\int_t^T \int_r^T \nabla g_u(r,\partial U_u^r, \partial V_u^r, U_u^r,V_u^r,\Yc_u,\Zc_u)\d u \d r&= \int_t^T \int_t^u \nabla g_u(r,\partial U_u^r, \partial V_u^r, U_u^r, V_u^r,\Yc_u,\Zc_u)\d r \d u\\
&= \int_t^T   g_u(u,U_u^u,V_u^u,\Yc_u,\Zc_u)-  g_u(t, U_u^t,V_u^t,\Yc_u,\Zc_u) \d u.
\end{align*}

Similarly, given $\partial V\in \H^{2,2}$, the version of Fubini's theorem for stochastic integration, see \cite[Theorem 65]{protter2005stochastic}, yields
\begin{align*}
 \int_t^T \int_r^T {\partial V_u^r}^\t  \d X_u \d r=\int_t^T \int_t^u \partial {V_u^r}^\t \d r \d X_u  =\int_t^T  \big( V_u^u-V_u^t)^\t  \d X_u \d r.
\end{align*}
Now, by \Cref{Lemma:partialU} we have that, $\P\as$, for any $t\in [0,T]$
\[ \int_t^T \int_r^T \d \partial M_u^r \d r= \int_t^T\big( \partial M_T^r-\partial M_r^r \big)\d r=\widetilde M_T-\widetilde M_t +M_t^t-M^t_T. \] 
We are left to verify $\widetilde M\in \M^2$. Indeed, note that $\widetilde M_0=0$ and

\begin{enumerate}[label=$(\roman*)$, ref=.$(\roman*)$,wide, labelwidth=!, labelindent=0pt]
\item $\widetilde M$ has c\`adl\`ag paths. This follows from the fact $(M_t^t)_{t\in [0,T]}$ has c\`adl\`ag path. Indeed for $t\in [0,T]$
\begin{align*}
M_t^t= M_t^T-\int_t^T \partial M_t^r\d r, \; t\in [0,T];
\end{align*} \vspace{-1.5em}
\item $\widetilde M$ is a martingale. Indeed by \Cref{Lemma:partialU}, for $0\leq u\leq t\leq T$, $\P\as$
$$ \displaystyle  \E\big[ \widetilde M_t  | \Fc_u\big] =\E \big[M_t^t  | \Fc_u\big]-\int_0^u \partial M_r^r \d r- \int_u^t \E\big[ \partial M_r^r\big| \Fc_u\big] \d r= M_u^t -\int_0^u \partial M_r^r \d r- \int_u^t \partial M_u^r \d r=\widetilde M_u  ;$$ \vspace{-1.5em}
\item $\widetilde M$ is orthogonal to $X$. For $0\leq u\leq t\leq T$, $\P\as$
\begin{align*}
\E \big[ X_t \widetilde M_t \big|\Fc_u\big] & ={\color{black} \E\big[ X_tM_t^t\big|\Fc_u\big]}-\E\bigg[ X_t \int_0^u \partial M_r^r \d r\Big|\Fc_u\bigg] -\int_u^t \E\big[ \E\big[  X_t|\Fc_r\big] \partial M_r^r|\Fc_u\big] \d r\\
& = X_uM_u^t-X_u \int_0^u   \partial M_r^r  \d r -X_u \int_u^t  \partial M_u^r \d r\\
& = X_uM_u^t-X_u \int_0^u  \partial M_r^r  \d r -  X_u  M_u^t+X_u M_u^u\\
&= X_u\widetilde M_u.
\end{align*}
where in the second equality we used the fact $\int_0^u \partial M_r^r \d r$ is $\Fc_u$--measurable, the tower property and the orthogonality of {\color{black} $M^s$ and $X$} and of $\partial M^s$ and $X$ for $s\in [0,T]$. The third inequality follows from \Cref{Lemma:partialU}.
\item $\|\widetilde M\|_{\M^2}<\infty$,
\begin{align*}
\| \widetilde M\|_{\M^2}^2  =\E \bigg[ \bigg( M_T^T -\int_0^T \partial M_r^r \d r\bigg)^{\! 2}\bigg]\leq 2\Bigg( \E \big[[M^T]_T\big]+T\int_0^T \E\big[ [\partial M^r]_r\big]\d r\bigg)\leq 2\Big[ \|M^T\|^2_{\M^2}+T^2 \| \partial M\|_{\M^{2,2}}^2 \Big]<\infty.
\end{align*} 
\end{enumerate}\vspace{-1.5em}
\end{proof}
\subsection{{\it A priori} estimates} 
We now establish \emph{a priori} estimates for \eqref{Eq:systemBSDE}. To ease the readability let us recall
\begin{align*}\label{Eq:systemBSDE2}\tag{$\Sc$}
\begin{split}
\Yc_t&=\xi(T,X_{\cdot\wedge T})+\int_t^T h_r(X,\Yc_r,\Zc_r, U_r^r,V_r^r,\partial U_r^r )\d r-\int_t^T  \Zc_r^\t \d X_r-\int_t^T \d N_r,\; t \in[0,T], \\
U_t^s&=   \eta (s,X_{\cdot\wedge T})+\int_t^T  g_r(s,X,U_r^s,V_r^s, \Yc_r, \Zc_r) \d r-\int_t^T {V_r^s}^\t  \d X_r-\int_t^T \d M^s_r,\; (s,t)\in[0,T]^2,  \\
\partial U_t^s&=  \partial_s \eta (s,X_{\cdot\wedge T})+\int_t^T  \nabla g_r(s,X,\partial U_r^s,\partial V_r^s,U_r^s,V_r^s, \Yc_r, \Zc_r) \d r-\int_t^T\partial  {V_r^s}^\t  \d X_r-\int_t^T \d \partial M^s_r,\; (s,t)\in[0,T]^2.
\end{split}
\end{align*}

Let us introduce $(\Hc^o,\|\cdot \|_{\Hc^o})$ and $(\Hf^o,\|\cdot \|_{\Hf^o})$ where $\|\cdot \|_{\Hc^o}$ and $\|\cdot\|_{\Hf^o}$ denote the norms induced by
\begin{gather*}
 \Hc^o:=\L^2 \times \H^2 \times  {\M^2} \times \L^{2,2} \times \Ho \times {\M^{2,2}} \times\L^{2,2}  \times \H^{2,2} \times {\M^{2,2}} , \; 
  \Hf^o:=\L^2 \times \H^2 \times  {\M^2} \times \big(\L^{2,2} \times \H^{2,2}  \times {\M^{2,2}}\big)^2 ,
\end{gather*}
 
To obtain estimates between the difference of solutions, it is more convenient to work with norms defined by adding exponential weights. We recall, for instance, that for $c\in [0,\infty)$ the norm $\|\cdot \|_{\H^{2,c}}$ is given by 
\[
\|\Zc\|_{\H^{2,c}}^2=\E \bigg[  \int_0^T  \e^{ct}|\sigma^\t_t\Zc _t^s|^2 \d t  \bigg],
\]
and they are equivalent for any two values of $c$, since $[0,T]$ is compact. With this, we define the norm $\|\cdot \|_{\Hc^{o,c}}$. In the following, we take the customary approach of introducing arbitrary constants $C>0$ to our analysis. These constants will typically depend on the data of the problem, e.g. the Lipschitz constants and $T$ and on the value of $c$ unless otherwise stated.  \medskip

\begin{proposition}\label{Prop:aprioriestimates}
Let {\rm \Cref{AssumptionA}} hold and $(\Yc,\Zc,\Nc,U,V,M,\partial U,\partial V,\partial M)\in  \Hc^o$ satisfy \eqref{Eq:systemBSDE}. Then $(\Yc,U,\partial U) \in \S^2 \times \S^{2,2}\times \S^{2,2}$. Furthermore, there exists a constant $C>0$ such that for $\|\cdot \|_{\Hc}$ as in {\rm \Cref{Section:infdimsystem}}
\begin{align*}
\|(\Yc,\Zc,\Nc,U, V,M,\partial U,\partial V,\partial M) \|^2_{{\Hc}} \leq C  \underbrace{ \Big(  \|\xi \|^2_{\Lc^2}  + \|\eta \|^2_{\Lc^{2,2}}+\| \partial_s \eta \|^2_{\Lc^{2,2}} +\|\tilde h\|_{\L^{1,2}}^2 +\| \tilde g \|_{\L^{1,2,2}}^2+\| \nabla \tilde g \|_{\L^{1,2,2}}^2 \Big)}_{=:I_0^2} <\infty.
\end{align*}
\end{proposition}

\begin{proof}
We proceed in several steps. We recall that in light of \Cref{AssumptionA}, $\d t \otimes\d \P\ae$
\begin{align}\label{Eq:Lipasshg}
|  h_r(\Yc_r,\Zc_r,U_r^r, V^r_r,\partial U_r^r))|& \leq |  \tilde h|+L_{  h} \big(|\Yc_r|+|\sigma_r^\t \Zc_r|+ |U_r^r|+|\sigma_r^\t V_r^r|+|\partial U_r^r| \big),\notag \\
|  g_r(s,U_r^s, V^s_r,\Yc_r,\Zc_r))|& \leq |  \tilde g(s)|+L_{  g} \big(|U_r^s|+|\sigma_r^\t V_r^s|+|\Yc_r|+|\sigma_r^\t \Zc_r|\big),\\
|  \nabla g_r(s,\partial U_r^s, \partial V^s_r,U_r^s, V^s_r,\Yc_r,\Zc_r))|& \leq |  \nabla \tilde g_r(s)|+ L_{\partial_s g} \big( |U_r^s|+|\sigma_r^\t V_r^s|+|\Yc_r|+|\sigma_r^\t \Zc_r|\big) +L_{  g} \big(|\partial U_r^s|+|\sigma_r^\t \partial V_r^s| \big).\notag  
\end{align}

\textbf{Step $1$:} we start with auxiliary estimates. By Meyer--It\=o's formula for $\e^{\frac{c}2  t} |\partial U_t^s|$, see \citet*[Theorem 70]{protter2005stochastic}
\begin{align}\label{eq:eq1}
\begin{split}
&\e^{\frac{c}2  t}|\partial U_t^s|+ L_T^0 -\int_t^T \e^{\frac{c}2  r} \sgn( \partial U_r^s)\cdot  \partial {V_r^s}^\t  \d X_r -\int_t^T \e^{\frac{c}2  r-} \sgn( \partial U_{r-}^s)\cdot  \d \partial M_r^s \\
&\; =\e^{\frac{c}2  T}  |\partial_s \eta (s)|  + \int_t^T  \e^{\frac{c}2  r} \bigg(  \sgn(  \partial U_r^s)  \cdot \nabla g_r(s,\partial U_r^s,\partial V_r^s,U_r^s,V_r^s,\Yc_r,\Zc_r)-\frac{c}2  |\partial U_r^s| \bigg) \d r ,\; t\in[0,T],
\end{split}
\end{align} 
where $L^0:=L^0(\partial U^s)$ denotes the non--decreasing and pathwise--continuous local time of the semi--martingale $\partial U^s$ at $0$, see \cite[Chapter IV, pp. 216]{protter2005stochastic}. We also notice that for any $s\in [0,T]$ the last two terms on the left--hand side are martingales, recall that $\partial V^s\in \H^2$.\medskip

We now take conditional expectation with respect to $\Fc_t$ in \Cref{eq:eq1}. We may use \eqref{Eq:Lipasshg} and the fact $ L^0$ is non--decreasing to derive that for $t\in[0,T]$ and $c>2 L_{g}$
\begin{align}\label{Eq:ineqUst}
\e^{\frac{c}2 t}| \partial  U_t^s| & \leq  \E_t \bigg[ \e^{\frac{c}2 T} |\partial_s  \eta (s)|+\int_t^T \e^{\frac{c}2 r} \Big(|\partial  U_r^s| (L_{  g}-c/2)+ | \nabla  \tilde g_r(s)| + L_{  g}  |\sigma^\t_r   \partial V_r^s| \Big)  \d r \notag \\
& \hspace{3em} +\int_t^T  \e^{\frac{c}2 r} L_{ \partial_s g}\big ( |U_r^s|+ |\sigma^\t_r   V_r^s| + |\Yc_r|+ |\sigma^\t_r  \Zc_r|\big) \d r\bigg] \notag \\
& \leq  \E_t \bigg[ \e^{\frac{c}2 T} |\partial_s \eta (s)|+\int_t^T \e^{\frac{c}2 r} \Big(  | \nabla \tilde g_r(s)| +\bar L \big ( |\sigma^\t_r   \partial V_r^s|+|U_r^s|+ |\sigma^\t_r   V_r^s| + |\Yc_r|+ |\sigma^\t_r  \Zc_r|\big) \Big)  \d r \bigg].
\end{align}

where $\bar L :=\max\{ L_g,L_{\partial_s g}\}$. Squaring in \eqref{Eq:ineqUst}, we may use \eqref{Eq:ineqsquare} and Jensen's inequality to derive for $ t\in[0,T]$
\begin{align*}
\frac{\e^{ct}}{7} |\partial U_t^t|^2 & \leq   \E_t\bigg[ \e^{cT} |\partial_s \eta (t)|^2+ T \bar  L ^2 \int_t^T \e^{c r}   \big( |U_r^t|^2+|\Yc_r|^2 + |\sigma^\t_r   \partial V_r^t|^2 +|\sigma^\t_r V_r^t|^2 + |\sigma^\t_r\Zc_r|^2\big) \d r \bigg]\\
&\quad +\E_t\bigg[\bigg(\int_t^T \e^{\frac{c}2 r} | \nabla \tilde g_r(t)|\d r\bigg)^2\bigg].
\end{align*}

Integrating the previous expression and taking expectation, it follows from the tower property that for any $t\in[0,T]$
\begin{align*}
\frac{  1}7\E\bigg[\int_t^T \e^{cr}|\partial U_r^r|^2\d r\bigg]\leq &\ \E\bigg[ \int_t^T   \e^{cT} |\partial_s \eta (r)|^2\d r\bigg]+\E \bigg[ \int_t^T  \bigg(  \int_r^T \e^{\frac{c}2 u}| \nabla \tilde g_u(r)|\d u\bigg)^2    \d r\bigg] \\
& + T \bar L^2 \E \bigg[ \int_t^T  \int_r^T \e^{c u}   \big( |U_u^r|^2 + |\Yc_u|^2 + |\sigma^\t_u \partial V_u^r|^2+ |\sigma^\t_u V_u^r|^2 + |\sigma^\t_u\Zc_u|^2\big) \d u \d r\bigg]\\
\leq &\ T \sup_{r\in [0,T]} \bigg\{ \|  \e^{cT} |\partial_s \eta(r)|^2]\|_{\Lc^2}+  \E \bigg[ \bigg( \int_t^T \e^{\frac{c}2 u} | \nabla \tilde g_u(r)|\d u \bigg)^2 \bigg]    \bigg\}   \\
& + T^2 \bar L^2   \sup_{r\in [0,T]} \bigg\{ \E\bigg[   \int_t^T \e^{c u} \big( |U_u^r|^2 +|\sigma^\t_u V_u^r|^2+|\sigma^\t_u \partial V_u^r|^2\big) \d u\bigg] \bigg\}  \\
& +    T^2 \bar L^2 \E \bigg [ \int_t^T \e^{c u}   \big( |\Yc_u|^2 + |\sigma^\t_u\Zc_u|^2\big) \d u \bigg].
\end{align*}

Thus, we obtain there is $C_{\partial u}>0$ such that for any $c>2L_{ g}$ and $t\in[0,T]$
\begin{align}\label{Eq:ineqpartialUtt2}\begin{split}
 \frac{1}{C_{\partial u} } \E\bigg[ \int_t^T  \e^{cr}  |\partial U_r^r|^2  \d r\bigg] &  \leq  \e^{cT}\big(   \|\partial_s \eta \|_{\Lc^{2,2}}^2+ \| \nabla \tilde g\|^2_{\L^{1,2,2}} \big)  +  \E\bigg[   \int_t^T \e^{cr} \big( |\Yc_r|^2 + |\sigma^\t_r\Zc_r|^2\big)\d r\bigg]  \\
 & \quad + \sup_{s\in [0,T]} \E\bigg[ \int_t^T \e^{c r} \big(   |U_r^s|^2+ |\sigma_r^\t V_r^s|^2+  |\sigma_r^\t \partial V_r^s|^2\big) \d r\bigg].
\end{split}
\end{align}

Similarly, we may find $C_u>0$ such that for any $c>2L_{ g}$ and $t\in [0,T]$
\begin{align}\label{Eq:ineqUtt2}\begin{split}
\frac{1}{C_u} \E\bigg[\int_t^T \e^{cr} |U_r^r|^2  \d r \bigg] & \leq  \e^{cT}\big(  \|\eta \|_{\Lc^{2,2}}^2+   \| \tilde g\|_{\L^{1,2,2}}^2 \big)  +\E\bigg[ \int_t^T \e^{cr}\big( |\Yc_r|^2+ |\sigma_r^\t \Zc_r|^2\big) \d r \bigg]\\
& \quad + \sup_{s\in [0,T]} \E\bigg[ \int_t^T \e^{c r}  |\sigma_r^\t V_r^s|^2\d r\bigg]  .
\end{split}
\end{align}

We now estimate the term $V_t^t$. In light of \eqref{Eq:aprioriest:partialUV} and \eqref{Eq:ineqVtt2:0}, there exists $C>0$ such that for any $\eps>0$, $c>4 L_{g}$ and $t\in [0,T]$
\begin{align*}
\E\bigg[  \int_t^T\e^{cu} |\sigma^\t_u  V_u^u|^2 \d u\bigg]&  \leq \E\bigg[ \int_t^T \e^{cu} | \sigma^\t_u  V_u^t|^2 \d u\bigg] + \int_t^T\E\bigg[ \int_r^T\e^{cu} \big( \eps |\sigma_u^\t   V^r_u|^2+ \eps^{-1} |\sigma_u^\t \partial V^r_u |^2\big) \d u\bigg] \d r\\
& \leq  \sup_{r \in [0,T]} \E\bigg[ \int_t^T \e^{cu}  | \sigma^\t_u  V_u^r |^2 \d u\bigg]+ T   \sup_{r \in [0,T]} \E\bigg[\int_t^T \e^{cu} \big( \eps | \sigma^\t_u  V_u^r|^2+ \eps^{-1}  | \sigma^\t_u\partial V_u^r|^2\big) \d u\bigg] \\
& \leq  (1+\eps T)\sup_{r\in [0,T]} \E\bigg[ \int_t^T\e^{cu}   | \sigma^\t_u V_u^r|^2 \d u\bigg]  +T C\eps^{-1} \sup_{r\in [0,T]} \E\bigg[ \int_t^T \e^{cu}  \big(|U_u^r|^2 +|\sigma_u^\t  V_u^r|^2 \big) \d u \bigg]\ \\
& \hspace{1em}  +T C\eps^{-1}  \E\bigg[ \int_t^T\e^{cu}  \big(|\Yc_u|^2+|\sigma_u^\t  \Zc_u|^2 \big) \d u \bigg] \\
& \hspace{1em}  +T C\eps^{-1}\e^{cT}  \sup_{r\in [0,T]}    \E\bigg[\big| \partial_{s } \eta(r)\big|^2+\bigg( \int_t^T |\nabla \tilde g_u(r)|\d u\bigg)^2\bigg] 
\end{align*}
Thus, taking $\eps=TC$ we may find $C_v>0$ such that for any $c>4 L_{g}$ and $t\in [0,T]$ 
\begin{align}\label{Eq:ineqVtt2}\begin{split}
\frac{1}{C_v} \E\bigg[  \int_t^T \e^{cu}  |\sigma^\t_r V_r^r|^2 \d r \bigg]  & \leq  \e^{cT}\big(  \|\partial_{s} \eta\|_{\Lc^{2,2}}^2+ \| \nabla \tilde g\|_{\L^{1,2,2}}^2 \big)+  \E\bigg[ \int_t^T  \e^{cr} \big(|\Yc_r|^2+|\sigma_r^\t  \Zc_r|^2 \big) \d r \bigg]   \\
&\quad   + \sup_{s\in [0,T]} \E\bigg[ \int_t^T \e^{c r} \big(   |U_r^s|^2+ |\sigma_r^\t V_r^s|^2 \big) \d r\bigg].
\end{split}
\end{align}

We emphasise that the constants $(C_{\partial u},C_u, C_v)\in (0,\infty)^3$ depend only of the data of the problem and are universal for any value of $c> 4L_{g}$.\medskip

\textbf{Step $2$:} Let $s\in [0,T]$, we show that $(\Yc, U,\partial U) \in \S^2 \times \S^{2,2}\times \S^{2,2}$. To alleviate the notation we introduce
\[ \Yf:=( \Yc,U^s,\partial U^s) , \; \mathfrak{Z}:=( \Zc ,V^s,\partial V^s) ,\; \mathfrak{N}:=( \Nc ,M^s,\partial M^s),\]
whose elements we may denote with superscripts, e.g. $\Yc^1,$ $\Yc^2,$ $\Yc^3$ correspond to $\Yc,$ $U^s,$ $\partial U^s$.\medskip

By \eqref{Eq:ineqsquare} and \eqref{Eq:Lipasshg}, we obtain that there exists $C>0$, which may change value from line to line, such that
\begin{align*}
|U_t^s|^2 \leq&\ C \bigg(| \eta (s)|^2  + \bigg(\int_0^T|  \tilde g_r(s)|\d r \bigg)^{\!2}+\int_0^T \!  \big(  |U^s_r|^2+| \sigma_r^\t  V_r^s|^2 +|\Yc_r|^2+| \sigma_r^\t  \Zc_r|^2 \big) \d r+\bigg| \int_t^T   {V_r^{s}}^\t \d X_r+ \int_t^T   \d M_r^s\bigg|^{ 2} \bigg).
\end{align*}
We note that by Doob's inequality
\begin{align}\label{ControlSingleSI}
\E\Bigg[ \sup_{t\in [0,T]} \bigg| \int_0^t {V_r^s}^\t   \d X_r\bigg|^2\Bigg]\leq 4 \| V^s\|^2_{\H^2}.
\end{align}
Taking supremum over $t\in [0,T]$ and expectation we obtain for $s\in [0,T]$
\begin{align}\label{AuxAS2U}
\|U^s\|_{\S^2}^2 \leq C\big( \|\eta(s)\|_{\Lc^2}^2+ \|  \tilde g(s)\|^2_{\L^{1,2}}+  \| U^s\|^2_{\L^2}+ \|V^s\|^2_{\H^2}+\| \Yc\|^2_{\L^2}+ \| \Zc\|^2_{\H^2}+ \| M^s\|^2_{\M^2} \big)<\infty.
\end{align}

Similarly, we obtain that there exists $C>0$ such that for $s\in [0,T]$
\begin{align}\label{AuxAS2partialU}
\|\partial U^s\|_{\S^2}^2 \leq C\bigg( \|\partial_s \eta(s)\|_{\Lc^2}^2+ \| \nabla  \tilde g(s)\|^2_{\L^{1,2}}+ \sum_{ i=1}^3  \| \mathfrak{Y}^i \|^2_{\L^2}+\|\mathfrak{Z}^i \|^2_{\H^2}+   \| \partial M^s\|^2_{\M^2} \bigg)<\infty.
\end{align}

Given $(\eta,\partial_s \eta,  \tilde g, \partial \tilde g) \in (\Lc^{2,2})^2\times (\L^{1,2,2})^2$, $(U,\partial U, V,\partial V,M, \partial M)\in (\L^{2,2})^2\times (\H^{2,2})^2\times ( \M^{2,2})^2$, \eqref{AuxAS2U} and \eqref{AuxAS2partialU}, the mapping 
\[ ([0,T],\Bc([0,T])) \longrightarrow (\S^{2},\|\cdot \|_{ \S^{2}}): s \longmapsto \mathfrak{Y}^{is} \text{ is continuous for } i \in \{2,3 \}.\] 

Given $s\in [0,T]$, $\| \mathfrak{Y}^i \|_{\S^{2,2}}<\infty$. Consequently, $\mathfrak{Y}^i \in \S^{2,2}$ for $i \in \{ 2,3 \}$. Arguing as above we may also derive,
\begin{align*}
 | \Yc_t |^2 &\leq  C\bigg( |\xi |^2+\bigg(\int_0^T |\tilde h_r|\d r \bigg)^2  + \int_0^T \big( |\Yc_r|^2+|\sigma_r^\t \Zc_r|^2+  |U_r^r|^2+|\sigma_r^\t V_r^r|^2 + |\partial U_r^r|^2\big) \d r \\
 &\hspace*{2.5em} +\bigg| \int_t^T \Zc_r^\t \d X_r \bigg|^2+\bigg| \int_t^T \d N_r\bigg|^2 \bigg), 
\end{align*}
which in turn yields, in combination with \eqref{Eq:ineqpartialUtt2}, \eqref{Eq:ineqUtt2} and \eqref{Eq:ineqVtt2},
\begin{align}\label{AuxAS3}\begin{split}
\|\Yc\|_{\S^2}^2 \leq &\ C\big( I_0^2 +\|\Yc\|_{\L^2}^2+ \|  \Zc\|^2_{\H^2} +\|U\|_{\L^{2,2}}^2+ \|V\|^2_{\H^{2, 2}}+ \|\partial V\|^2_{\H^{2, 2}}+ \| N\|^2_{\M^2}  \big)<\infty.
\end{split}
\end{align}

Finally, taking sup over $s\in [0,T]$ and adding together \eqref{AuxAS3} \eqref{AuxAS2U} and \eqref{AuxAS2partialU} we obtain
\begin{align}\label{Eq:NormSestimate}
\begin{split}
\|\Yc\|_{\S^2}^2+ \|U^s\|_{\S^{2,2}}^2 + \|\partial U\|_{\S^{2,2}}^2\leq &\ C\big( I_0^2 +\| ( \Yc,\Zc,\Nc,U,V,M,\partial U,\partial V, \partial M)\|_{\Hc^o}\big).
\end{split}
\end{align}

\textbf{Step $3$:} We obtain the estimate of the norm. To ease the notation we introduce 
\begin{align*}
h_r:=h_r( \Yc_r, \Zc_r,U_r^r,V_r^r,\partial U_r^r), \; g_r(s):= g_r(s,U_r^s, V^s_r,\Yc_r,  \Zc_r), \nabla g_r(s):= \nabla g_r(s,\partial U_r^s, \partial V^s_r,U_r^s, V^s_r,\Yc_r,  \Zc_r).
\end{align*} 

By applying It\=o's formula to $\e^{ct}\big( |\Yc_t|^2+|U_t^s|^2+|\partial U_t^s|^2\big)$ we obtain, $\P\as$
\begin{align*}
& \sum_{ i=1}^3  \e^{ct} \big|\mathfrak{Y}^i_t\big|^2 +   \int_t^T \e^{cr}  \big| \sigma_r^\t \mathfrak{Z}^i_r\big|^2  \d r+\int_t^T \e^{cr-}   \d \big[{\mathfrak{N}^i}\big]_r +\Mf_t-\Mf_T \\
& = \e^{cT}(| \xi(T)|^2 +|\eta (s)|^2+|\partial_s \eta (s)|^2) +\int_t^T \e^{cr}\big(2  \Yc_r \cdot  h_r + 2 U_r^s \cdot g_r(s) +2 \partial U_r^s \cdot  \nabla g_r(s)-c ( |\Yc_r|^2+|U_r^s|^2+ |\partial U_r^s|^2) \big)\d r,
\end{align*}
where we used the orthogonality of $X$ and both $M^s$ and $N$, and we introduced the martingale
\begin{align*}
\Mf_t:=2 \sum_{i=1}^3 \int_0^t\e^{cr} \mathfrak{Y}^i_r\cdot {\mathfrak{Z}^i_r}^\t\,  \d X_r+\int_0^t \e^{cr-} \mathfrak{Y}^i_{r-} \cdot \d\,  \mathfrak{N}^i_r .
\end{align*}

Indeed, the Burkholder--Davis--Gundy inequality in combination with the fact that $(\Yc,U^s,\partial U^s)\in (\S^2)^3$ shows that $\Mf$ is uniformly integrable, and consequently a true martingale. Moreover, 
we insist on the fact that the integrals with respect to $N$, $M^s$ and $\partial M^s$ account for possible jumps, see \cite[Lemma 4.24]{jacod2003limit}.\medskip

Moreover, as $(\Yc, U, \partial U) \in \S^2 \times ( \S^{2,2})^2$, from \eqref{Eq:Lipasshg} and with Young's inequality we obtain that for any $(\eps,\tilde \eps)\in (0,\infty)^2$, there is $C(\tilde \eps) \in (0,\infty)$ such that the left--hand side above is smaller than
\begin{align*}
& \leq   \e^{cT}\big(| \xi|^2+|\eta (s)|^2+|\partial_s \eta (s)|^2 \big) +   \int_t^T  \e^{cr} ( C(\tilde \eps)-c) \big(   |{\Yc}_r|^2+ |U_r^s|^2 + |\partial U_r^s|^2  \big) \d r\\
&\quad  +\int_t^T \tilde \eps   \e^{cr} \bigg(    | \sigma_r^\t \Zc_r|^2+  | \sigma_r^\t V_r^s|^2+  | \sigma_r^\t \partial V_r^s|^2  + \frac{1}{ C_u} |U_r^r|^2 + \frac{1}{ C_v}|V_r^r|^2 +\frac{1}{C_{\partial u}}|\partial U_r^r|^2  \bigg) \d r \\
&\quad  +\eps \sum_{i=1}^3  \sup_{r\in [0,T]} \e^{cr} \big|\mathfrak{Y}^i_t\big|^2  +\eps^{-1} \bigg(\int_0^T \e^{\frac{c}2 r} |\tilde h_r |\d r\bigg)^2 + \eps^{-1}\bigg(\int_0^T \e^{\frac{c}2 r}|   \tilde g_r(s)|\d r \bigg)^2 + \eps^{-1}\bigg(\int_0^T \e^{\frac{c}2 r}|  \nabla \tilde g_r(s)|\d r \bigg)^2  ,
\end{align*}

with $(C_{\partial u},C_u,C_v)$ as in \eqref{Eq:ineqpartialUtt2}--\eqref{Eq:ineqVtt2}. Taking expectation and letting $c>4L_g$, we find there is $C>0$ such that
\begin{align*}
&\E\bigg[\sum_{ i=1}^3  \e^{ct} \big|\mathfrak{Y}^i_t\big|^2 +   \int_t^T \e^{cr}  \big| \sigma_r^\t \mathfrak{Z}^i_r\big|^2  \d r+\int_t^T \e^{cr-}   \d \big[{\mathfrak{N}^i}\big]_r  \bigg] \\
&\ \leq  \E\bigg[ \int_t^T  \e^{cr}    |{\Yc}_r|^2 ( C(\tilde \eps) -c) \d r\bigg] + \sup_{s\in [0,T]}  \E\bigg[ \int_t^T  \e^{cr} |U_r^s|^2 (  C(\tilde \eps) -c)\d r\bigg]+\sup_{s\in [0,T]} \E\bigg[\int_t^T  \e^{cr}  |\partial U_r^s|^2 (  C(\tilde \eps) -c)  \d r\bigg]\\
&\hspace*{1em}  +(1+\eps^{-1}+\tilde \eps )C I_0^2 + \tilde \eps C \Big( \|\Zc\|^2_{\H^{2,c}}+ \| V\|^2_{\H^{2,2,c}}+ \|\partial  V\|^2_{\H^{2,2,c}} \Big) + \eps  C  \big(\|\Yc\|_{\S^{2,c}}^2+\|U\|^2_{\S^{2,2,c}}+\|\partial U\|^2_{\S^{2,2,c}}  \big).
\end{align*}
We then let $\tilde \eps= 1/(2^4C)$, $c\geq  \max\{ 4L_g, C(\tilde \eps)\}$, and take $\sup$ over $t\in [0,T]$ (resp. $(s,t)\in [0,T]^2$) to each term on the left side separately. Adding these terms up we find there is $C>0$, such that for any $\eps>0$
\begin{align}\label{AuxAS4}
\begin{split}
\frac{1}T\|(\Yc,\Zc,\Nc,U,V,M,\partial U,\partial V, \partial M)\|_{\Hf^{o}} & \leq \,   \sup_{t\in [0,T]}\E\big[ \e^{ct} |\Yc_t|^2 \big]+   \sup_{(s,t)\in [0,T]^2}\E\big[ \e^{ct} |U_t^s|^2\big]+   \sup_{(s,t)\in [0,T]^2}\E\big[ \e^{ct} |\partial U_t^s|^2\big]\\[0.3em]
&\hspace*{1em} +  \|  \Zc\|^2_{\H^2} +  \|   V\|^2_{\H^{2,2}}+  \|   \partial V\|^2_{\H^{2,2}}+ \| \Nc\|_{\M^2}^2+ \| M\|_{\M^{2,2}}^2+ \|\partial  M\|_{\M^{2,2}}^2  \\[0.3em]
& \leq   (1+\eps^{-1} )C I_0^2   +  \eps C  \big(\|\Yc\|_{\S^2}^2+\|U\|^2_{\S^{2,2}}+\|\partial U\|^2_{\S^{2,2}}  \big).
\end{split}
\end{align}

We can the use \eqref{AuxAS4} back in \eqref{Eq:NormSestimate} to find $ \eps \in (0,\infty)$ small enough so that\footnote{Recall the norm $\|\cdot\|_{\Hf}$ is the norm induced by the space $\Hf$ as defined in \Cref{Remark:wpsystem}\ref{Remark:wpsystem:ii}}
\begin{align*}
\|(\Yc,\Zc,\Nc,U, V,M,\partial U,\partial V,\partial M) \|^2_{{\Hf}} \leq C I_0^2 
\end{align*}
The result in terms of the norm $\|\cdot \|_{\Hc}$ follows from \eqref{Eq:ineqVtt2}.
\end{proof}

\begin{proposition}\label{Prop:aprioriestimatesdif}
Let $(\xi^i ,\eta^i,\partial_s \eta^i)\in  \Lc^2\times (\Lc^{2,2})^2$ and $(h^i, g^i, \partial_s g^i )$ for $i\in\{1,2\}$ satisfy {\rm\Cref{AssumptionA}} and suppose in addition that $\Hf^i \in  \Hc^o$ is a solution to \eqref{Eq:systemBSDE} with coefficients $(\xi^i,h^i,\eta^i, g^i,\partial_s \eta^i, \nabla g^i)$, $i\in\{1,2\}$. Then
\begin{align*}
\| \delta \Hf  \|^2_{\Hc} \leq &\ C\Big(  \|\delta \xi\big\|^2_{\Lc^2} +\|\delta \eta \big\|^2_{\Lc^{2,2}}+\|\delta \partial \eta \big\|^2_{\Lc^{2,2}} +  \|\delta_1 h \|_{\L^{1,2}}^2+ \|\delta_1  g \|_{\L^{1,2,2}}^2+\|\delta_1  \nabla g \|_{\L^{1,2,2}} \Big),
\end{align*}
where for $\varphi\in\{\Yc,\Zc,\Nc,U,V,M,\partial U,\partial V,\partial M,\xi,\eta,\partial_s \eta \}$ and $\Phi\in \{h,g, \nabla g\}$
\[
\delta \varphi:= \varphi^1-\varphi^2,\; \text{\rm and, }\; \delta_1 \Phi_t:= \Phi^1_t( \Yc_r^1, \Zc^1_t,U^{1t}_t,V^{1t}_t)- \Phi^2_t( \Yc_r^1, \Zc^1_t,U^{1t}_t,V^{1t}_t),\; \d t\otimes \d \P \ae \text{ in }[0,T]\times \Xc.
\]
\end{proposition}

\begin{proof}
Note that by the Lipschitz assumption on $h$ and $g$ there exist bounded processes with appropriate dimensions $(\alpha^i, \beta^i,\gamma^i, \eps^i)$, $i\in\{1,2,3\}$, $\rho$ and $\varrho$ such that 
\begin{align*}
\delta \Yc_t =&\ \delta \xi(T) + \int_t^T \big( \delta_1 h_r +\gamma_r^1  \delta \Yc_r + \alpha^{1\t}_r  \sigma_r^\t \delta \Zc_r  +  \beta^1_r \delta U_r^r+\eps^{1\t}_r  \sigma_r^\t \delta V_r^r\big) \d r - \int_t^T  \delta \Zc_r^\t \d X_r - \int_t^T \d \delta N_r,\\
\delta U_t^s=&\ \delta \eta (s)+ \int_t^T  \big(\delta_1  {g}_r(s) + \beta_r^2  \delta U_r^s  +\eps^{2\t}_r \sigma_r^\t \delta V_r^s   + \gamma_r^2  \delta \Yc_r + \alpha_r^{2\t}    \sigma_r^\t \delta \Zc_r \big) \d r - \int_t^T  \delta  V_r^{s\t} \d X_r - \int_t^T  \d \delta  M^s_r\\
\delta \partial U_t^s=&\ \delta \partial_s \eta (s)+ \int_t^T  \big(\delta_1  \nabla g_r(s) +\rho_r \delta  \partial U^s_r+\varrho_r \delta  \partial V^s_r+ \beta_r^3  \delta U_r^s  +\eps^{3\t}_r \sigma_r^\t \delta V_r^s  +  \gamma_r^3  \delta \Yc_r + \alpha_r^{3\t}    \sigma_r^\t \delta \Zc_r \big) \d r\\
&\; - \int_t^T  \delta  \partial {V_r^{s} }^\t \d X_r - \int_t^T  \d\delta   \partial   M^s_r.
\end{align*}
We can therefore apply Lemma \ref{Prop:aprioriestimates} and the result follows.
\end{proof}

When the data of the system is chosen so as to study the class of {\rm type--I BSVIEs} considered in {\rm \Cref{Section:BSVIE}}, our approach can be specialised so as to enlarge the initial space and  simplify the {\it a priori} estimates obtained in {\rm \Cref{Prop:aprioriestimates}}. We let $(\Hc^{\star,o}, \|\cdot\|_{\Hc^{\star,o}})$ be
\[
 \Hc^{\star,o}:= \L^{2}\times \H^{2} \times {\M^{2}} \times \L^{2,2} \times \H^{2,2}  \times {\M^{2,2}} , 
 \]
 \[
   \|(\Yc,\Zc,\Nc,Y,Z,N)  \|^2_{\Hc^{\star,o}}:=  \|\Yc\|_{\L^{2}}^2 + \|\Zc\|_{\H^{2}}^2+\|\Nc \|_{\M^{2}}^2+  \|Y\|_{\L^{2,2}}^2 + \|Z\|_{\H^{2,2}}^2+\|N \|_{\M^{2,2}}^2.
\]

\begin{proposition}\label{Prop:aprioriest:simplify}
Let {\rm \Cref{Assumption:SystemBSVIEwp}} hold and consider $(\Yc,\Zc,\Nc,Y,Z,N)\in \Hc^{\star,o}$ solution to
\begin{align}\label{Eq:prop:aprioriest:simplify}\tag{$\Sc_f^o$}
\begin{split}
\Yc_t&=\xi(T,X)+\int_t^T\big( f_r(r,X,\Yc_r,\Zc_r,Y_r^r)- \partial Y_r^r\big) \d r-\int_t^T \Zc_r^\t X_r-\int_t^T \Nc_r,\; t\in [0,T], \\
Y_t^s&=   \xi (s,X)+\int_t^T  f_r(s,X,Y_r^s,Z_r^s, \Yc_r) \d r-\int_t^T {Z_r^s}^\t  \d X_r-\int_t^T \d N^s_r,\; (s,t)\in[0,T]^2.
\end{split}
\end{align}
with $\partial Y$ given as in {\rm \Cref{Lemma:partialU}}. Then $(\Yc,Y)\in \S^2\times \S^{2,2}$ and there exist $C>0$ such that
\begin{align*}
\| (\Yc,\Zc,\Nc,Y,Z,N) \|_{\Hc^\star}^2 \leq C  \Big(   \|\xi \|^2_{\Lc^{2,2}} +\|\tilde f \|_{\L^{1,2,2}}^2 \Big) <\infty.
\end{align*}
\end{proposition}

\begin{proof}
We first note that $Y\in \S^{2,2}$ follows as in \Cref{Prop:aprioriestimates}. Thus, in light of {\rm \Cref{Lemma:partialU}}, there exists $(\partial Y,\partial Z, \partial N)\in \S^{2,2}\times \H^{2,2}\times \M^{2,2}$ solution to the {\rm BSDE} with data $(\partial_s \xi, \partial f)$, and, $\big( (Y_t^t)_{t\in [0,T]},(Z_t^t)_{t\in [0,T]}\big)\in\S^2\times \H^2$ are well--defined. Moreover, $\partial Y_r^r$ is well--defined as an element of $\L^{1,2}$, i.e. $\d t \otimes \d \P\ae$ on $[0,T]\times \Xc$, as a consequence of the path--wise continuity of $\partial Y_\cdot^s$. With this, we conclude $\Yc\in \S^2$.\medskip

Let us now note that given $(Y_t^t)_{t\in[0,T]} ,(Z_t^t)_{t\in[0,T]}$ and $(\partial Y_t^t)_{t\in[0,T]}$, the first equation, being a Lipschitz {\rm BSDE}, admits a unique solution $(\Yc,\Zc,\Nc)$. In addition, {\color{black} for $\widetilde N$} as in {\rm \Cref{Lemma:intdiagpartialU}} we obtain
\begin{align}\label{Eq:prop:esimplfy0}
Y_t^t  =Y_T^T+\int_t^T  \big( f_r(r,Y_r^r,Z_r^r,\Yc_r)- \partial Y_r^r\big) \d r -\int_t^T {Z_r^r}^\t  \d X_r-\int_t^T {\color{black} \d \widetilde N_r} ,\; t\in [0,T],\; \P\as
\end{align}
Thus, $\big(\Yc_\cdot,\Zc_\cdot,\Nc_\cdot\big)=\big((Y_t^t)_{t\in[0,T]},(Z_t^t)_{t\in[0,T]},{\color{black} (\widetilde N_t)_{t\in[0,T]}}\big)$ in $\S^2\times \H^2\times \M^2$ and we obtain
\begin{align}\label{Eq:prop:esimplfy}
\Yc_t  = \xi(t)+\int_t^T  f_r(t,Y_r^t,Z_r^t,\Yc_r) \d r -\int_t^T {Z_r^t}^\t  \d X_r-\int_t^T \d \Nc_r^t ,\; t\in [0,T],\; \P\as
\end{align}
With this equation we can simplify our estimates. Let us introduce the system
\begin{align*}\tag{$\Ac$}
\begin{split}
\Yc_t  &= \xi(t)+\int_t^T  f_r(t,Y_r^t,Z_r^t,\Yc_r) \d r -\int_t^T {Z_r^t}^\t  \d X_r-\int_t^T \d  \Nc_r^t ,\; t\in [0,T], \\
Y_t^s&=   \xi (s)+\int_t^T  f_r(s,Y_r^s,Z_r^s, \Yc_r) \d r-\int_t^T {Z_r^s}^\t  \d X_r-\int_t^T \d N^s_r,\; (s,t)\in[0,T]^2.
\end{split}
\end{align*}
Then, following the same reasoning of \Cref{Prop:aprioriestimates}, i.e. applying It\=o's formula to $\e^{ct}\big( |\Yc_t|+ | Y_t^s|) $ in combination with Young's inequality, we obtain there is $C>0$ such that $\| (\Yc,\Zc,\Nc,Y,Z,N) \|_{\Hc^\star}^2 \leq C  \big(   \|\xi \|^2_{\Lc^{2,2}} +\|\tilde f \|_{\L^{1,2,2}}^2 \big) <\infty.$
\end{proof}

\subsection{Well--posedness}

Before we present the proof of \Cref{Thm:wp} we recall that in light of \Cref{Prop:aprioriestimates} and \Cref{Prop:aprioriestimatesdif} once the result is obtained for $\Hc^o$ the existence of a unique solution in $\Hc$ follows immediately.
\begin{proof}[Proof of {\rm \Cref{Thm:wp}}]
Note that uniqueness follows from \Cref{Prop:aprioriestimatesdif}. To show existence, let us define the map
\begin{align*}
\Tf:  \Hc^o &\longrightarrow  \Hc^o\\
(y,z,n,u,v,m,{\rm u}, {\rm v}, {\rm m} )& \longmapsto (Y,Z,N,U,V,M,\partial U, \partial V, \partial M),
\end{align*} 
with $(\Yc,\Zc,\Nc,U,V,M,\partial U, \partial V, \partial M)$ given by
\begin{align*}
\Yc_t&=\xi(T,X_{\cdot\wedge T})+\int_t^T h_r(X,y_r, z_r, U_r^r, V_r^r,\partial U_r^r)\d r-\int_t^T \Zc_r^\t d X_r-\int_t^T \d  N_r,\\
U_t^s&= \eta (s,X_{\cdot\wedge,T})+\int_t^T  g_r(s,X,U_r^s,V_r^s, y_r, z_r) \d r-\int_t^T  {V_r^s}^\t \d X_r-\int_t^T \d  M^s_r,\\
\partial U_t^s&= \partial_s \eta (s,X_{\cdot\wedge,T})+\int_t^T   \nabla g_r(s,X,\partial U_r^s,\partial V_r^s,U_r^s,V_r^s, y_r, z_r) \d r-\int_t^T  \partial {V_r^s}^\t  \d X_r-\int_t^T \d  \partial M^s_r.
\end{align*}

\medskip
{\bf Step $1$:} We first show $\Tf$ is well defined. Let $(y,z,n,u,v,m,{\rm u}, {\rm v}, {\rm m} )\in  \Hc^o$.

\begin{enumerate}[label=$(\roman*)$, ref=.$(\roman*)$,wide, labelwidth=!, labelindent=0pt]

\item Let us first consider the tuples $(U,V,M)$ and $(\partial U,\partial V,\partial M)$. Let us first consider the second equation. Given $(y,z)\in \S^2\times \H^2$ and \Cref{AssumptionA}, this equation is a standard Lipschitz BSDE whose well--posedness follows by classical arguments, see \cite{zhang2017backward, el1997backward}. This yields $(U^s,V^s,M^s)\in \S^2\times \H^2\times \M^2$ for all $s\in [0,T]$.\medskip

Let us argue the continuity of $([0,T],\Bc([0,T])) \longrightarrow (\L^{2},\|\cdot \|_{ \L^{2}}): s \longmapsto U^s $. Let $(s_n)_n\subseteq [0,T], s_n \xrightarrow{n\to\infty} s_0\in[0,T]$ and define for $\varphi\in \{U,V,\eta\}$, $\Delta \varphi^n:=\varphi^{s_n}-\varphi^{s_0}$. From the classic stability result for BSDEs we obtain that there is $C>0$ such that
\begin{align*}
\E\bigg[\int_0^T | \Delta U_t^n|^2\d r \bigg]\leq 2T\Big( \|\Delta {\eta}^n\|^2_{\Lc^2} +T L_{ g}^2 \big(\rho_{ g}^2(|s_n-s_0|)\Big).
\end{align*}

We conclude $\|U\|_{\L^{2,2}}< \infty$ and $U\in {\L^{2,2}}$. Given $(U^s,V^2s)\in \L^2\times \H^2$ together with $(y,z)\in \S^2\times \H^2$, the argument for $(\partial U^s, \partial V^s, \partial M^s)$ for fixed $s\in [0,T]$ is identical. \medskip

\item  We now show that $(V,\partial V , M,  \partial M) \in (\H^{2,2})^2 \times ({\M^{2,2}})^2$. Again, the argument for $(\partial V,\partial M)$ is completely analogous. Applying It\=o's formula to $|U_r^s|^2$ we obtain
\begin{align*}
|U_t^s|^2+\int_t^T | \sigma_r^\t V_r^s|^2 \d r+\int_t^T \d [ M^s]_r=&\ |\eta (s)|^2 +2\int_t^T  U_r^s  \cdot g_r(s,U_r^s, V^s_r,y_r,  z_r)\d r-2\int_t^T U^s_r\cdot {V^s_r}^\t  \d X_r\\
&-2\int_t^T U^s_{r-} \cdot \d M_r^s.
\end{align*}
First note \ $U^s \in \S^{2}$ guarantees that the last two terms are true martingale for any $s\in[0,T]$. To show that $([0,T],\Bc([0,T])) \longrightarrow (\H^{2},\|\cdot \|_{ \H^{2}})\, \big($resp. $ (\M^{2},\|\cdot \|_{ \M^{2}})\big) \, : s \longmapsto V^s \, \big($resp. $ M^s\big)$ is continuous, let $(s_n)_n\subseteq [0,T],$ $s_n\xrightarrow{n\to\infty} s_0\in[0,T]$. We then deduce there is $C>0$ such that
\begin{align*}
\E\bigg[\int_0^T | \sigma_r^\t \Delta V_r^n|^2\d r+ [ \Delta M^s]_T  \bigg]\leq C\Big(\|\Delta  \eta\|_{\Lc^2}^2 +\rho_{ g}^2(|s_n-s_0|) \Big),
\end{align*}
and, likewise, we obtain
\begin{align*}
\sup_{s\in [0,T]} \E\bigg[\int_0^T | \sigma_r^\t V_r^s|^2\d r +  [ M^s]_T \bigg]\leq  C\Big(\|\eta \|_{\Lc^{2,2}}^2 +\| \tilde g\|^2_{\L^{1,2}} \Big)<\infty.
\end{align*}

Since the first term on the right--hand side is finite from \Cref{AssumptionA}, we obtain $\|V\|_{\H^{2,2}}+\|M\|_{\M^{2,2}}<\infty$.\medskip

We are left to argue for the process $\Vc:=(V_t^t)_{t\in [0,T]}$. This follows noticing that we may apply \Cref{Lemma:partialU} to
\begin{align*}
U_t^s&= \eta (s,X_{\cdot\wedge,T})+\int_t^T  g_r(s,X,U_r^s,V_r^s, y_r, z_r) \d r-\int_t^T  {V_r^s}^\t \d X_r-\int_t^T \d  M^s_r,\\
\partial U_t^s&= \partial_s \eta (s,X_{\cdot\wedge,T})+\int_t^T   \nabla g_r(s,X,\partial U_r^s,\partial V_r^s,U_r^s,V_r^s, y_r, z_r) \d r-\int_t^T  \partial {V_r^s}^\t  \d X_r-\int_t^T \d  \partial M^s_r,
\end{align*}
and obtain $\Vc$ is well--defined and $\|\Vc\|_{\H^2}<\infty$. We conclude $\Vc\in \H^2$ and thus $V\in \Ho$.

\item We derive an auxiliary estimate. Recall $(U,V,M,\partial U,\partial V, \partial M)$ satisfy \eqref{Eq:constraintspace}. Now, in light of \Cref{AssumptionA} and $(ii)$, we may find, as in {\bf Step $1$} in \Cref{Prop:aprioriestimates}, a universal constant $C>0$ such that for any $c> 4L_{ g}$ and $t\in [0,T]$
\begin{align}\label{Eq:ineqUVtt2wp}
\begin{split}
\frac{1}C \E\bigg[ \int_t^T  \e^{cr}\big( |U_r^r|^2+  |\partial U_r^r|^2+ |\sigma^\t_r V_r^r|^2\big)  \d r\bigg] & \leq   \e^{cT}  \big(  \|\eta \|_{\Lc^{2,2}}^2+   \| \tilde g\|_{\L^{1,2}}^2+  \|\partial_s \eta \|_{\Lc^{2,2}}^2+ \| \nabla \tilde g\|^2_{\L^{1,2,2}} \big) \\
 & \hspace{1em}  +  \E\bigg[   \int_t^T \e^{cr} \big( |y_r|^2 + |\sigma^\t_rz_r|^2\big)\d r\bigg]
\end{split}
\end{align}

\item We argue for the tuple $(\Yc,\Zc,\Nc)$, notice that
\begin{align*}
\widetilde \Yc_t:= \E \bigg[\xi(T)+\int_0^T h_r(y_r, z_r, U_r^r, V_r^r, \partial U_r^r )\d r\bigg|\Fc_t \bigg], \text{ is a square integrable $\F$--martingale}.
\end{align*}

Indeed, under \Cref{AssumptionA}, $h$ is uniformly Lipschitz in $(y,z,u)$, so \eqref{Eq:ineqUVtt2wp} yields
\begin{align*}
\E\big[|\widetilde \Yc_t|^2\big]&\leq 6 \bigg( \|\xi \|_{\Lc^2}^2+\|\tilde h\|_{\L^{1,2}}^2+ T L_h^2 \bigg( \|y\|^2_{\H^2}+\|z\|^2_{\H^2}+\E\bigg[\int_0^T \big( |U_r^r|^2 +|V_r^r|^2+|\partial U_r^r|^2  \big)\d r \bigg]  \bigg) \bigg) <\infty, \;  \forall t\in [0,T]
\end{align*}\vspace*{-1em}

Integrating the above expression, Fubini's theorem implies that $\widetilde \Yc\in \L^2$, thus the the predictable martingale representation property for local martingales guarantees the existence of a unique $(\Zc, \Nc) \in  \H^2\times \M^2$ such that $(\Yc,\Zc,\Nc)$ satisfies the correct dynamics and Doob's inequality implies $\Yc\in \S^2$, where 
\[ \Yc
:=\widetilde \Yc -\E\bigg[\int_0^\cdot h_r(y_r,z_r,U_r^r,V_r^r,\partial U_r^r)\d r\bigg].\]\vspace*{-1em}

All together, we have shown that $\phi(y,z,n,u,v,m,\partial u,\partial v,\partial m)\in \Hc^o$.
\end{enumerate}

\medskip
{\bf Step $2$:} We show $\Tf$ is a contraction under the equivalent norm $\|\cdot\|_{  \Hc^{o,c}}$, for some $c>0$ large enough. Let $(y^i,z^i,n^i,u^i,v^i,m^i) \in   \Hc^o$, $\mathfrak{h}^i=\phi(y^i,z^i,n^i,u^i,v^i,m^i,\partial u^i,\partial v^i,\partial m^i)$ for $i\in \{1,2\}$. We first note that by \Cref{Lemma:intdiagpartialU}
\begin{align*}
\Uc^{i}_t=\eta(T)+\int_t^T\big(g_r(r,\Uc^{i}_r,\Vc^{i}_r,y^i_r,z^i_r)-\partial \Uc^{i}_r\big)\d r-\int_t^T \Vc^{i}_r\d X_r -\int_t^T \d\widetilde \Mc^{i}_r, \; t\in [0,T],\; \P\as,
\end{align*}
where $(\Uc^i,\Vc^i,\widetilde\Mc^i, \partial \Uc^i):=\big( (U_t^{it})_{t\in [0,T]},(V_{t}^{it})_{t\in [0,T]},{\color{black} (\widetilde M_t^{i})_{t\in [0,T]}},(\partial U_t^{it})_{t\in [0,T]}\big)$ {\color{black} and $\widetilde M^i$ as in \Cref{Lemma:intdiagpartialU}} for $i\in \{1,2\}$.\medskip

To ease the readability we define
\begin{align*}
 \delta h_r& := h_r(y_r^1,z_r^1,\Uc_r^{1},\Vc_r^{1},\partial \Uc_r^{1})-h_r(y_r^2,  z_r^2,\Uc_r^{2},\Vc_r^{2},\partial \Uc_r^{2 }),\\
 \delta \hat g_r & := g_r(r,\Uc_r^{1},  \Vc_r^{1},y_r^1,  z_r^1)-\partial \Uc_r^{1}-  g_r(r,\Uc_r^{2}, \Vc_r^{2} ,y_r^2,  z_r^2)+\partial \Uc_r^{2},\\
 \delta g_r(s) & := g_r(s,U_r^{1s},  V_r^{1s},y_r^1,  z_r^1)- g_r(s,U_r^{2s}, V_r^{2s} ,y_r^2,  z_r^2),\\
 \delta  \nabla g_r(s)& :=  \nabla g_r(s,\partial {U_r^s}^{1} ,\partial {V_r^s}^{1},{U_r^s}^1,  {V_r^s}^1,{y_r^s}^1,  z_r^1)-  \nabla g_r(s,\partial {U_r^s}^{2} ,\partial {V_r^s}^{2},{U_r^s}^2, {V_r^s}^2,{y_r^s}^2,  z_r^2),
\end{align*}
and
\[ \delta \Yf:=( \delta \Yc,\delta \Uc,\delta U^s,\delta \partial U^s) , \; \delta \mathfrak{Z}:=( \delta \Zc ,\delta \Vc ,\delta V^s,\delta \partial V^s) ,\; \delta \mathfrak{N}:=( \delta \Nc , {\color{black}\delta \widetilde\Mc},\delta M^s,\delta \partial M^s),\]
whose elements we may denote with superscripts, e.g. $\delta \Yf^1,$ $\delta \Yf^2,$ $\delta \Yf^3,$ $\delta \Yf^4$ correspond to $\delta \Yc,$ $\delta \Uc,$ $\delta U^s,$ $\delta\partial U^s$.
\begin{enumerate}[label=$(\roman*)$, ref=.$(\roman*)$,wide, labelwidth=!, labelindent=0pt]
\item In light of \eqref{Eq:constraintspace}, as in {\bf Step $1$} in \Cref{Prop:aprioriestimates}, we may find that for $c> 4L_g$ there exists a universal constant $C\in (0,\infty)$ such that for $t\in [0,T]$
\begin{align}\label{Eq:ineqUVtt2wpcont}
\begin{split}
  \E\bigg[\int_t^T \e^{cr}\big( |\delta \Uc_r|^2  +  |\delta \partial \Uc_r|^2 +   |\sigma^\t_r \delta \Vc_r|^2\big) \d r\bigg] \leq  C   \E\bigg[   \int_t^T \e^{cr} \big( |\delta y_r|^2 + |\sigma^\t_r \delta z_r|^2\big)\d r\bigg].
\end{split}
\end{align}

\item Applying It\=o's formula to $\e^{cr}\big(|\delta \Yc_r|^2+|\delta \Uc_r|^2+|\delta U_r^s|^2+|\delta \partial U_r^s|^2\big)$ and noticing that $(\delta \Yc_T,  \delta \Uc_T,\delta U_T,\delta \partial U_T)=(0,0,0,0)$ we obtain
\begin{align*}
&  \sum_{ i=1}^4  \e^{ct} \big|\delta \mathfrak{Y}^i_t\big|^2 +   \int_t^T \e^{cr}  \big| \sigma_r^\t \delta \mathfrak{Z}^i_r\big|^2  \d r+\int_t^T \e^{cr-}   \d \big[\delta \mathfrak{N}^i \big]_r + \widetilde \Mf_T^s-\widetilde \Mf_t^s\\
& =\int_t^T \e^{cr}\Big(2 \delta Y_r\cdot  \delta h_r + 2\delta U_r^s\cdot   \delta  g_r(s)+ 2\delta \partial U_r^s\cdot   \delta   \nabla g_r(s)  - c\big( |\delta Y_r|^2+ |\delta U_r^s|^2+ |\delta \partial U_r^s|^2\big) \Big)\d r,
\end{align*}

where $ \widetilde \Mf_t^s= 2 \sum_{i=1}^4 \int_0^t\e^{cr} \delta \mathfrak{Y}^i_r\cdot \delta {\mathfrak{Z}^i_r}^\t \d X_r+\int_0^t \e^{cr-} \delta \mathfrak{Y}^i_{r-} \cdot \d\,  \delta \mathfrak{N}^i_r .$ Again, the fact that $(\delta \Yc,\delta \Uc,\delta U,\delta \partial U) \in (\S^{2})^2\times (\S^{2,2})^2$ guarantees, via the Burkholder--Davis--Gundy inequality, that $\widetilde \Mf^s$ is a uniformly integrable martingale, and thus a true martingale for all $s\in [0,T]$.\medskip

Additionally, under \Cref{AssumptionA}\ref{AssumptionA:ii} and \ref{AssumptionA}\ref{AssumptionA:iii}, $\d t \otimes\d \P\ae$
\begin{align*}\begin{split}
|\delta h_r | & \leq L_h (|\delta y_r|+| \sigma_t^\t \delta z_r| + |\delta \Uc_r|+|\sigma_r^\t \delta \Vc_r|+ |\delta \partial \Uc_r^r|),\\
|\delta \hat{g}_r | & \leq L_g (|\delta y_r|+| \sigma_t^\t \delta z_r| + |\delta \Uc_r|+|\sigma_r^\t \delta \Vc_r|)+ |\delta \partial \Uc_r^r|,\\
|  \delta g_r(s)| & \leq L_{  g} \big(|\delta U_r^s|+| \sigma_r^\t \delta V^s_r|+|\delta y_r|+ |\sigma^\t_r \delta z_r|\big),\\
|  \delta  \nabla  g_r(s)| & \leq L_{ \partial_s  g} \big(|\delta  U_r^s|+| \sigma_r^\t \delta   V^s_r|+|\delta y_r|+ |\sigma^\t_r \delta z_r|\big)+L_g( |\delta \partial U_r^s| + |\sigma_r^\t \delta \partial V_r^s|).\end{split}
\end{align*}
In turn, this implies together with Young's inequality and \eqref{Eq:ineqUVtt2wpcont}, that for any $c>4 L_g$ there exists a universal constant $C>0$ such that for any $\eps>0$
\begin{align*}
&\ \E\bigg[ \sum_{ i=1}^4  \e^{ct} \big|\delta \mathfrak{Y}^i_t\big|^2 +   \int_t^T \e^{cr}  \big| \sigma_r^\t \delta \mathfrak{Z}^i_r\big|^2  \d r+\int_t^T \e^{cr-}   \d \big[ \delta \mathfrak{N}^i\big]_r	 \bigg]\\
&\  \leq \E\bigg[ \int_t^T \e^{cr}\Big (  \big (|\delta \Yc_r|^2+|\delta \Uc_r|^2+|\delta  U_r^s|^2+|\delta \partial U_r^s|^2\big) (C\eps^{-1} - c) + \eps\big(|\delta y_r|^2+ |\sigma^\t_r \delta z_r|^2+  |\delta \Uc_r|+  |\sigma_r^\t \delta \Vc_r|+  |\delta \partial \Uc_r| \big)\Big) \d r  \\
&\  \leq \E\bigg[ \int_t^T \e^{cr} \big(|\delta \Yc_r|^2+|\delta \Uc_r|^2\big) (C\eps^{-1} - c) \d r\bigg] +\sup_{s\in [0,T]} \E\bigg[ \int_t^T \e^{cr} |\delta  U_r^s|^2(C\eps^{-1} - c)\d r\bigg] \\
&\ \hspace{1em}+ \sup_{s\in [0,T]} \E\bigg[ \int_t^T \e^{cr} |\delta \partial U_r^s|^2(C\eps^{-1} - c) \d r\bigg]  +   \eps C \E\bigg[ \int_0^T \e^{cr}  \big(|\delta y_r|^2+ |\sigma^\t_r \delta z_r|^2  \big) \d r \bigg] ,
\end{align*}
where in the second inequality $C$ is appropriately updated. Choosing $\eps=Cc^{-1}$ we obtain
\begin{align*}
\ \E\bigg[ \sum_{ i=1}^4  \e^{ct} \big|\delta \mathfrak{Y}^i_t\big|^2 +   \int_t^T \e^{cr}  \big| \sigma_r^\t \delta \mathfrak{Z}^i_r\big|^2  \d r+\int_t^T \e^{cr-}   \d \big[ \delta \mathfrak{N}^i\big]_r \bigg]  \leq \frac{C}{c} \Big(\|\delta y\|_{\L^{2,c}} +\|\delta z\|_{\H^{2,c}} \Big).
\end{align*}
Which yields
\[ \| \delta \mathfrak{h}\|_{  \Hc^{o,c}}^2\leq \frac{C}{c} \Big(\|\delta y\|_{\L^{2,c}} +\|\delta z\|_{\H^{2,c}}\Big).\]
\end{enumerate}
We conclude $\Tf$ has a fixed point as it is a contraction for $c$ large enough.\end{proof}
\begin{remark}\label{Remark:contractionsimplify}
We remark that in the previous proof we never used the fact that $v\in \Ho$. This is, knowing that for the input $v$ in the contraction its diagonal process $(v_t^t)_{t\in [0,T]}$ is a element of $\H^2$ was not required.
\end{remark}

{\small 
\bibliography{BibliographyCamilo}}

\end{document}